%% file: onbflm.tex
\documentclass[10pt,a4paper]{amsart} 

\usepackage[utf8]{inputenc}
\usepackage[T1]{fontenc}
\usepackage{amsmath, amssymb}
\usepackage{amsthm}
\usepackage{mathtools}   

\usepackage{xspace} 
\usepackage{xfrac}  
\usepackage{ifthen} 
\usepackage{xifthen} 
\usepackage{array} 
\usepackage{booktabs} 

\usepackage{enumitem} 
\setenumerate[0]{label=(\roman*)}   

\numberwithin{equation}{section}

\usepackage{hyperref}
\hypersetup{
    unicode=true,          
    pdftitle={},    
    pdfauthor={He Weikun},     
    pdfsubject={},   
    pdfkeywords={}{}, 
    pdfborder={0 0 0}
}

\renewcommand{\phi}{\varphi} 

\newcommand\dash{\nobreakdash-\hspace{0pt}}

\newtheorem{thm}{Theorem}[section]
\newtheorem{coro}[thm]{Corollary}
\newtheorem{prop}[thm]{Proposition}
\newtheorem{lemm}[thm]{Lemma}
\newtheorem{conj}[thm]{Conjecture}
\theoremstyle{definition}
\newtheorem{déf}[thm]{Definition}

\newtheorem{exemp}[thm]{Example}

\newcommand{\N}{{\mathbb N}}
\newcommand{\Z}{{\mathbb Z}}
\newcommand{\R}{{\mathbb R}}
\renewcommand{\C}{{\mathbb C}}
\newcommand{\Q}{{\mathbb Q}}
\newcommand{\dd}{{\,\mathrm{d}}}

\newcommand{\ub}{\mathbf{u}}
\newcommand{\vb}{\mathbf{v}}
\newcommand{\wb}{\mathbf{w}}

\DeclareMathOperator{\End}{End}

\DeclareMathOperator{\SL}{SL}
\DeclareMathOperator{\GL}{GL}

\DeclareMathOperator{\SO}{SO}
\DeclareMathOperator{\Sp}{Sp}

\DeclareMathOperator{\Nbd}{Nbd}        
\DeclareMathOperator{\Gr}{Gr}        

\DeclareMathOperator{\im}{im}
\DeclareMathOperator{\rk}{rk}

\DeclareMathOperator{\Span}{Span}
\DeclareMathOperator{\Supp}{Supp}

\DeclareMathOperator{\vol}{vol}

\DeclareMathOperator{\diag}{diag}

\DeclareMathOperator{\dang}{d_\measuredangle}        
\DeclareMathOperator{\dH}{d_{\mathcal H}}            

\newcommand{\abs}[1]{\lvert#1\rvert}    
\newcommand{\norme}[1]{\left\lVert#1\right\rVert}   
\newcommand{\norm}[1]{\lVert#1\rVert}   

\newcommand{\sg}[1]{\left<#1\right>}               

\newcommand{\Prob}[2][]{\ifthenelse{\equal{#1}{}}   
            {\mathbb{P}}               
            {\mathbb{P}_{#1}}\bigl[#2\bigr]}

\newcommand{\Espr}[2][]{\ifthenelse{\equal{#1}{}}   
            {\mathbb{E}}               
            {\mathbb{E}_{#1}}\bigl[#2\bigr]}

\newcommand{\muex}[1][]{\ifthenelse{\isempty{#1}}   
            {\mu_{\mathrm{ex}}}               
            {\mu_{\mathrm{ex}}^{#1}}}

\DeclareMathOperator{\Ncov}{\mathcal{N}}

\DeclareMathOperator{\Lms}{\mathcal{L}}
\newcommand{\alphaini}{\alpha_\mathrm{ini}}
\newcommand{\alphahigh}{\alpha_\mathrm{high}}
\newcommand{\alphainc}{\alpha_\mathrm{inc}}
\newcommand{\Glen}{\mathcal{G}_\mathrm{len}}
\newcommand{\Gstat}{\mathcal{G}_\mathrm{stat}}
\newcommand{\Gvol}{\mathcal{G}_\mathrm{vol}}
\newcommand{\Gproj}{\mathcal{G}_\mathrm{proj}}

\renewcommand{\bullet}{\boldsymbol{\,\cdot\,}}

\begin{document}
\title[Random walks with a Schubert condition]{Random walks on linear groups satisfying a Schubert condition}

\author{Weikun He}
\thanks{The author is supported by ERC grant ErgComNum 682150.}
\address{Einstein Institute of Mathematics, The Hebrew University of Jerusalem, Jerusalem 91904, Israel.}
\email{weikun.he@mail.huji.ac.il}

\begin{abstract}
We study random walks on $\GL_d(\R)$ whose proximal dimension $r$ is larger than $1$ and whose limit set in the Grassmannian $\Gr_{r,d}(\R)$ is not contained any Schubert variety. 
These random walks, without being proximal, behave in many ways like proximal ones.
Among other results, we establish a Hölder-type regularity for the stationary measure on the Grassmannian associated to these random walks.
Using this and a generalization of Bourgain's discretized projection theorem, we prove that the proximality assumption in the Bourgain-Furman-Lindenstrauss-Mozes theorem can be relaxed to this Schubert condition. 

\end{abstract}

\maketitle

\input{intro}

\input{propS}

\input{grass}
\input{check}

\bibliographystyle{abbrv} 
\bibliography{mybib}

\end{document}

%% file: intro.tex
\section{Introduction}
Let $d \geq 2$ and let $\mu$ be a Borel probability measure on $\GL_d(\R)$. Let $\Gamma_\mu$ denote the closed semisubgroup generated by the support of $\mu$. 
The random walk on $\GL_d(\R)$ associated to $\mu$ is $(g_n \dotsm g_1)_{n\geq 0}$ where $(g_n)_{n \geq 1}$ is a sequence of independent and identically distributed random variables distributed according to $\mu$. Thus, the distribution of the random walk at time $n\geq 0$ is $\mu^{*n}$, the multiplicative convolution of $\mu$ with itself $n$ times. 

The study of asymptotic behaviors of these random walks, known as the theory of random matrix products, dates back to the 60's.
In this theory, a condition called proximality (also known as contraction) plays an important role.
In this article, we define a property that can be seen as a weak version of the proximality. 
The aim is then to find and prove, under this weaker assumption, results analogous to those which are already known under the proximality assumption.


Let us start by defining this property which we will call \eqref{eq:propertyS} in this article. In order to do so, recall some definition. The \emph{proximal dimension} of subsemigroup $\Gamma \subset \GL_d(\R)$ is 
\[r_\Gamma = \min \bigl\{ \rk \pi \mid \pi \in \overline{\R \Gamma} \setminus \{0\}\bigr\},\]
where $\overline{\R \Gamma}$ denotes the closure in $\End(\R^d)$ of the set of all elements  of the form $\lambda g$ with $\lambda \in \R$ and $g \in \Gamma$.
Since this notion of proximal dimension depends on the specific embedding of $\Gamma$ into some $\GL_d$, it is better to refer to this quantity $r_\Gamma$ as the proximal dimension of the representation $\R^d$ of $\Gamma$ or as the proximal dimension of the action of $\Gamma$ on $\R^d$. 

Thus, $\Gamma$ is proximal if and only if its proximal dimension is equal to $1$. 

We define
\[\Pi_\Gamma = \{\pi \in \overline{\R \Gamma} \mid \rk \pi = r_\Gamma\}.\]
Let $\Gr(r_\Gamma,d)$ denote the Grassmannian of $r_\Gamma$\dash{}dimensional linear subspaces of $\R^d$. The \emph{limit set of $\Gamma$ in the Grassmannian} is defined as 
\[\Lms_\Gamma = \{ \im \pi \in \Gr(r_\Gamma,d) \mid \pi \in \Pi_\Gamma \}.\]

\begin{déf}
We say that $\Gamma$ has property \eqref{eq:propertyS} if its limit set in $\Gr(r_\Gamma,d)$ is not contained in any proper Schubert variety. Equivalently, 
\begin{equation}
\label{eq:propertyS}
\tag{S} \forall W \in \Gr(d-r_\Gamma,d),\, \exists V \in \Lms_\Gamma,\, V \cap W = \{0\}.
\end{equation}
\end{déf}

For example, if the action of $\Gamma$ on $\R^d$ is irreducible and proximal then \eqref{eq:propertyS} is automatically satisfied.

Let $G$ denote the Zariski closure of $\Gamma$ in $\GL_d(\R)$, or in other words, the set of $\R$\dash{}rational points of the Zariski closure of $\Gamma$.
A fundamental result of Gol'dshe\u{\i}d-Margulis~\cite{GoldsheidMargulis} (see also \cite[Lemma 6.23]{BenoistQuint}) asserts that $r_G = r_\Gamma$. In particular, $\Gamma$ is proximal if and only if $G$ is.
We will prove in Lemma~\ref{lm:SisZariski} that $\Gamma$ has property \eqref{eq:propertyS} if and only if $G$ has.

Also in Section~\ref{sc:propS} we will see examples of non-proximal semigroups having the property \eqref{eq:propertyS}.

\subsection{Random walk on the Grassmannian}
Consider the action of $\Gamma = \Gamma_\mu$ on the Grassmannian $\Gr(r_\Gamma,d)$. Given a starting point $V \in \Gr(r_\Gamma,d)$, we then have a corresponding random walk on the Grassmannian: $(g_n \dotsm g_1V)_{n\geq 0}$.

A classical result due to Furstenberg~\cite{Furstenberg1973} asserts that if $\Gamma_\mu$ acts strongly irreducibly and proximally on $\R^d$ then there is a unique $\mu$-stationary Borel probability measure on the projective space $\mathbb{P}(\R^d)$.
We prove in Proposition~\ref{pr:uniquenu} that there is a unique $\mu$-stationary Borel probability measure on $\Gr(r_\Gamma,d)$, provided that $\Gamma$ acts strongly irreducibly on $\R^d$ and satisfies \eqref{eq:propertyS}.

The following proposition is a large deviation inequality about the probability that the random walk in $\Gr(r_\Gamma,d)$ falls into a Schubert variety. 

For $V \in \Gr(r,d)$ and $W \in \Gr(d-r,d)$, define
\[\dang(V,W) = \abs{\det(v_1,\dotsc,v_r, w_1,\dotsc,w_{d-r})}\]
where $(v_1,\dotsc,v_r)$ is an orthonormal basis of $V$ and $(w_1,\dotsc,w_{d-r})$ a basis of $W$.

\begin{prop}
\label{pr:RWonGr1}
Assume that $\mu$ has finite exponential moment, $\Gamma$ acts strongly irreducibly on $\R^d$ and satisfies \eqref{eq:propertyS}.
Then for any $\omega > 0$, there is $c > 0$ and $l_0 \geq 1$ such that for all $n \geq l \geq l_0$, the following holds for any $V \in \Gr(r_\Gamma,d)$ and any $W \in \Gr(d-r_\Gamma,d)$,
\[\mu^{*n} \{ g \in \Gamma \mid \dang(g V,W) \leq e^{-\omega l} \} \leq e^{-cl}.\]
\end{prop}

Roughly speaking, this result says that under assumption \eqref{eq:propertyS}, the random walk on $\Gr(r_\Gamma,d)$ does not concentrate in neighborhoods of any proper Schubert variety. 
Here again, if $\Gamma$ is proximal, then this estimate is already known~\cite[Lemma 4.5]{BFLM}. 
We will state a reformulation of this special case below as Theorem~\ref{thm:LargeD}\ref{it:LargeDmc}. 
In fact, we will use this special case as an ingredient in the proof of Proposition~\ref{pr:RWonGr1}.

From another point of view, Proposition~\ref{pr:RWonGr1} can be seen as a special case of the question how the random walk on a linear group $G$ escapes proper subvarieties of $G$. This general question is considered by Aoun in 
\cite{Aoun}. However the main result there (\cite[Theorem 1.2]{Aoun}) is contidional to the Zariski closure of $\Gamma$ being split over $\R$, which only allow to treat the proximal case since all representations of a $\R$-split $\R$-group are proximal.

In Corollary~\ref{cr:regularity} we prove that this result implies a Hölder-type regularity for the $\mu$-stationary measure on $\Gr(r_\Gamma,d)$. Again the proximal case ($r_\Gamma=1$)  is already known and is due to Guivarc'h~\cite[Théorème 7']{Guivarch}.

\subsection{Random walk on the torus}
Now assume that $\mu$ is supported on $\SL_d(\Z)$ and consider the action of $\Gamma = \Gamma_\mu$ on the $d$\dash{}dimensional torus $\mathbb{T}^d = \R^d/\Z^d$. Given a starting point $x_0 \in \mathbb{T}^d$, we then have a corresponding random walk on the torus: $(g_n \dotsm g_1x_0)_{n\geq 0}$. We are interested in the equidistribution of the measure $\mu^{*n}*\delta_{x_0}$, i.e. the distribution of $g_n \dotsm g_1x_0$.

Let us recall the statement of the Bourgain-Furman-Lindenstrauss-Mozes theorem. Let $d \geq 2$ be an integer.

\begin{thm}[Bourgain-Furman-Lindenstrauss-Mozes~\cite{BFLM}]
\label{thm:BFLM}
Let $\mu$ be a probability measure on $\SL_d(\Z)$ with finite exponential moment. Let $\Gamma$ denote the subsemigroup generated by $\Supp(\mu)$. Assume
\begin{itemize}
\item[{\upshape (I)}] the action of $\Gamma$ on $\R^d$ is strongly irreducible;
\item[{\upshape (P)}] the action of $\Gamma$ on $\R^d$ is proximal.
\end{itemize}
Let $\lambda_{1,\mu}$ denote the top Lyapunov exponent of $\mu$. Then for any $0 < \lambda < \lambda_{1,\mu}$ there is a constant $C = C(\mu, \lambda)$ so that if for a point $x \in \mathbb{T}^d$ the measure $\nu_n = \mu^{*n}*\delta_{x}$ satisfies that for some $a \in \Z^d \setminus \{0\}$,
\[ \abs{\hat\nu_n(a)} > t \text{ with } n > C\log\frac{2\norm{a}}{t},\] 
then $x$ admits a rational approximation $\frac{p}{q}$ for $p \in \Z^d$ and $q\in \Z_+$ satisfying
\[ \norme{x -\frac{p}{q}} < e^{-\lambda n} \text{ and } q < \Bigl(\frac{2\norm{a}}{t}\Bigr)^C.\]
\end{thm}

For more background and consequences of this result, we refer the readers to the original article~\cite{BFLM}. As pointed out by the authors, the assumption~{\upshape (P)} is a technical condition. It is widely believed that the theorem should hold without this condition. Moreover, the work of Benoist and Quint~\cite[Corollary 1.4]{BQII} suggests the following conjecture.
\begin{conj}
The assumptions {\upshape (I)} and {\upshape (P)} on $\Gamma$ in Theorem~\ref{thm:BFLM} can be replaced by the assumption that the Zariski closure of $\Gamma$ is semisimple, Zariski connected and with no compact factor and acts irreducibly on $\Q^d$. 
\end{conj}

In the present article, we present the following partial result towards this conjecture. 
\begin{thm}
\label{thm:main}
Theorem~\ref{thm:BFLM} still holds when the assumption~{\upshape (P)} is replaced by the assumption that $\Gamma$ satisfies \eqref{eq:propertyS} and that $r_\Gamma$ divides $d$.
\end{thm}

For example, on account of Proposition~\ref{pr:Cgroups} and Example~\ref{ex:SLdH}, the Bourgain-Furman-Lindenstrauss-Mozes Theorem holds if the Zariski closure of $\Gamma$ in $\GL_d(\R)$ is one of the following.
\begin{enumerate}
\item $\SL(d/2,\C)$ embedded in $\SL(d,\R)$, $d \geq 4$ is a multiple of $2$,
\item $\SO(d/2,\C)$ embedded in $\SL(d,\R)$, $d \geq 6$ is a multiple of $2$, 
\item $\Sp(d/4,\C)$ embedded in $\SL(d,\R)$, $d \geq 4$ is a multiple of $4$. 
\item $\SL(d/4,\mathbb{H})$ embedded in $\SL(d,\R)$, $d \geq 8$ is a multiple of $4$. 
\end{enumerate}

There are two inputs in proving this result. 
The first one is the non\dash{}concentration estimate Proposition~\ref{pr:RWonGr1} above.   
The second one is a higher rank discretized projection theorem proved in \cite{He_proj}, which we will state as Theorem~\ref{thm:proj} below. The latter is a generalization of the discretized projection theorem for projections to lines due to Bourgain~\cite{Bourgain2010}.

Having these two inputs available, there is no difficulty in adapting the original proof of Bourgain-Furman-Lindenstrauss-Mozes to the situation of Theorem~\ref{thm:main}. 
We will omit a large part of the details and instead only indicate places where attention needs to be payed.

The requirement of avoiding all proper Schubert varieties comes from the statement of the discretized projection theorem. However the discretized projection theorem proved in \cite{He_proj} is far from being optimal. Thus, if one proves a stronger projection theorem, one could expect to cover a larger class of non-proximal groups.

The assumption that the proximal dimension $r_\Gamma$ divides the dimension of the space $d$ is again a technical condition. It is essentially used to say that $d/r_\Gamma$ subspaces of dimension $r_\Gamma$ in general position are in direct sum and the sum is the entire $\R^d$.
Removing this assumption would make this article unnecessarily long without conceptual novelty.

\subsection{Organisation of the article}
In Section~\ref{sc:propS}, we will start by proving some facts about the property~\eqref{eq:propertyS} then give both examples of groups with and without \eqref{eq:propertyS}. In particular, we prove in Lemma~\ref{lm:SisZariski} that a subsemigroup of $\GL_d(\R)$ satisfies the property~\eqref{eq:propertyS} if and only if its Zariski closure in $\GL_d(\R)$ does. We prove in Proposition~\ref{pr:Cgroups} that groups obtained by restriction of scalars from $\C$ to $\R$ have property~\eqref{eq:propertyS}.

Section~\ref{sc:grass} is devoted to random walks on the Grassmannian. The main goal is to prove Proposition~\ref{pr:RWonGr1}, the non-concentration estimate for neighborhoods of proper Schubert varieties. Actually, we will prove a more detailed version of it in Proposition~\ref{pr:RWonGr}. Then we show how it can be interpreted as a regularity result of the stationary measure on the Grassmannian (Corollary~\ref{cr:regularity}).

In section~\ref{sc:check}, we will highlight several places in the proof of Theorem~\ref{thm:main}. This part is not self-contained, since much of the proof is just routine check and a large part of the details can be found in the original article~\cite{BFLM}.

\subsection{Notation convention}
The notation for the Grassmannian variety $\Gr(\bullet)$, the projective space $\mathbb{P}(\bullet)$ and the space of endomorphisms $\End(\bullet)$ are with respect to the linear structure over the field $\R$.
We will specify with a subscript when we are dealing with linear structure over another field.
For instance, $\End_\C(\C^d)$ is the space of $\C$-linear endomorphisms of $\C^d$ while $\End(\C^d)$ denotes the space of $\R$-linear endomorphisms of the underlying $\R$-linear space of $\C^d$.

For nonzero vector $x \in \R^d$, $\bar{x}$ denotes the line $\R x \in \mathbb{P}(\R^d)$.

\subsection*{Acknowledgement}
I would like to thank Emmanuel Breuillard, Hillel Furstenberg, Elon Lindenstrauss, Shahar Mozes and Péter Varj\'u for enlightening conversations.

%% file: propS.tex
\section{The property \texorpdfstring{\eqref{eq:propertyS}}{(S)}}
\label{sc:propS}

Let $\Gamma$ be a subsemigroup of $\GL_d(\R)$.
Assume that $\Gamma$ acts irreducibly on $\R^d$. 
We write $r = r_\Gamma$ to denote its proximal dimension and $\Lms_\Gamma$ its limit set in the Grassmannian $\Gr(r,d)$ (for the definitions see the introduction).
Unless state otherwise, the Grassmannian $\Gr(r,d)$ is endowed with its topology of differential manifold.

Since the action of $\Gamma$ on $\R^d$ is irreducible,  \cite[Lemma 4.2]{BenoistQuint} asserts that $\Lms_\Gamma$ is a minimal closed set in $\Gr(r,d)$ under the action of $\Gamma$.
When $\Gamma$ is proximal, it is the unique minimal $\Gamma$\dash{}invariant set in $\Gr(r,d)$.
When $\Gamma$ is not proximal, there could be several disjoint minimal close sets (cf. \cite[Remark 4.4]{BenoistQuint}).

Recall that we defined the property \eqref{eq:propertyS} as
\begin{equation}\label{eq:imSI}
\forall W \in \Gr(d - r, d),\; \exists \pi \in \Pi_\Gamma,\quad \im \pi \cap W= \{0\}.
\end{equation}

\subsection{Properties} 
We collect some basic properties about \eqref{eq:propertyS}.

\begin{lemm}
\label{lm:SisZariski}
Let $\Gamma$ be a subsemigroup of $\GL(\R^d)$ acting irreducibly on $\R^d$. Let $G$ be the Zariski closure of $\Gamma$. Then $\Gamma$ satisfies \eqref{eq:propertyS} if and only if $G$ satisfies \eqref{eq:propertyS}. Moreover, $\Gamma$ satisfies \eqref{eq:propertyS} if and only if
\begin{equation}\label{eq:kerSI}
\forall V \in \Gr(r,d),\; \exists \pi \in \Pi_\Gamma,\quad V \cap \ker \pi = \{0\}.
\end{equation}
\end{lemm}

\begin{proof}
Let $G$ be the Zariski closure of $\Gamma$ in $\GL_d(\R)$. By Gold'she\u{\i}d-Margulis~\cite{GoldsheidMargulis}, $G$ has the same proximal dimension as $\Gamma$ : $r_G = r_\Gamma = r$. 


We first establish the equivalence between \eqref{eq:imSI} and 
\begin{equation}\label{eq:imSI_G}
\forall W \in \Gr(d-r,d),\; \exists \pi \in \Pi_G,\quad \im \pi \cap W = \{0\}.
\end{equation}
More precisely, they are equivalent to each of the following conditions.
\begin{enumerate}
\item \label{it:imSI1} $\forall W \in \Gr(d-r,d)$, $\forall \pi \in \Pi_\Gamma$, $\exists g \in \Gamma$ such that $g \im \pi \cap W = \{0\}$.
\item \label{it:imSI2} $\forall W \in \Gr(d-r,d)$, $\forall \pi \in \Pi_\Gamma$, $\exists g \in G$ such that $g \im \pi \cap W = \{0\}$.
\item \label{it:imSI3} $\forall W \in \Gr(d-r,d)$, $\forall \pi \in \Pi_G$, $\exists g \in G$ such that $g \im \pi \cap W = \{0\}$.
\end{enumerate}
Obviously, \ref{it:imSI1}$\implies$\eqref{eq:imSI}.
To see \eqref{eq:imSI}$\implies$\ref{it:imSI1}, we assume that $W \in \Gr(d-r,d)$ and $\pi' \in \Pi_\Gamma$ such that $\im \pi' \cap W = \{0\}$. Let $\pi \in \Pi_\Gamma$ be another element. Since $\Lms_\Gamma$ is a minimal $\Gamma$-invariant subset of $\Gr(r,d)$, there is a sequence $(g_n) \in \Gamma^\N$ such that $g_n \im \pi \to \im \pi'$. Thus, $g_n \im \pi \cap W = \{0\}$ for $n$ large enough.
The same argument applied to $G$ instead of $\Gamma$ gives the equivalence between \eqref{eq:imSI_G} and \ref{it:imSI3}.
Since $\Pi_\Gamma \subset \Pi_G$ we have immediately the implications \eqref{eq:imSI}$\implies$\eqref{eq:imSI_G} and \ref{it:imSI3}$\implies$\ref{it:imSI2}. 
Finally, \ref{it:imSI2}$\implies$\ref{it:imSI1} because for fixed subspaces $\im \pi$ and $W$, the set of $g$ such that $g \im \pi \cap W \neq \{0\}$ is a Zariski closed subset of $\GL_d(\R)$.

It remains to show the equivalence between \eqref{eq:imSI} and \eqref{eq:kerSI}. For $f \in \End(\R^d)$ denote by $f^*$ its adjoint with respect to the usual Euclidean structure of $\R^d$.  Observe that for all $\pi \in \End(\R^d)$, we have $\ker \pi^* = (\im \pi)^\perp$ and that $V \cap W = \{0\}$ if and only if $V^\perp \cap W^\perp = \{0\}$ for all $V \in \Gr(r,d)$ and all $W \in \Gr(d-r,d)$. Consequently, the argument above applied to $\Gamma^*$ shows the equivalence between \eqref{eq:kerSI} and
\begin{equation}\label{eq:kerSI_G}
\forall V \in \Gr(r,d),\; \exists \pi \in \Pi_G,\quad V \cap \ker \pi = \{0\}.
\end{equation}

Using the same observation, we deduce \eqref{eq:imSI_G} $\iff$ \eqref{eq:kerSI_G} immediately for the special case when $G$ is self-adjoint (i.e. $G^* = G$). 
In the general case, since $G$ is an algebraic group acting irreducibly on $\R^d$, by a result of Mostow~\cite{Mostow}, there is a positive definite symmetric bilinear form on $\R^d$ with respect to which $G$ is self-adjoint. Thus we are back to the special case.
\end{proof}

\begin{lemm} 
\label{lm:LimSI}
If $\Gamma$ acts irreducibly on $\R^d$ and satisfies \eqref{eq:propertyS}, then $\Lms_\Gamma$ is the unique minimal $\Gamma$-invariant subset in $\Gr(r_\Gamma,d)$.
\end{lemm}

\begin{proof}
Let $V \in \Gr(r_\Gamma,d)$. By \eqref{eq:kerSI}, there is $\pi \in \Pi_\Gamma$ such that $V \cap \ker \pi = \{0\}$. Let $(\lambda_n) \in \R^\N$ and $(g_n) \in \Gamma^\N$ be sequences such that $\lambda_n g_n \to \pi$. Then $g_n V \to \im \pi$ as $n \to +\infty$. Therefore $\im \pi \in \overline{\Gamma V}$ and hence $\Lms_\Gamma \subset \overline{\Gamma V}$. It follows that $\Lms_\Gamma$ is the unique minimal $\Gamma$-invariant subset in $\Gr(r_\Gamma,d)$.
\end{proof}

\begin{lemm}
\label{lm:Sforfindex}
Let $\Gamma$ be a subgroup of $\GL(\R^d)$.
\begin{enumerate}
\item \label{it:Sforf1} Let $\Gamma'$ be a finite index subgroup of $\Gamma$. Then $\Gamma$ satisfies \eqref{eq:propertyS} if and only if $\Gamma'$ does. 
\item \label{it:Sforf2} If $\Gamma$ acts strongly irreducibly on $\R^d$ and satisfies \eqref{eq:propertyS} then $\Lms_\Gamma$ is not contained in any finite union of proper Schubert varieties in $\Gr(r,d)$.
\item \label{it:Sforf3} If $\Gamma$ acts strongly irreducibly on $\R^d$ and satisfies \eqref{eq:propertyS} then $\Gamma$ does not preserve any finite union of proper Schubert varieties in $\Gr(r,d)$.
\end{enumerate}
\end{lemm}

\begin{proof}
If $\Gamma' < \Gamma$ has finite index then there exists a finite set $F \subset \GL_d(\R)$ such that $\Gamma = \Gamma' F$. We deduce that $r_{\Gamma'} = r_\Gamma$ and $\Pi_\Gamma =   \Pi_{\Gamma'} F$ and hence $\Lms_\Gamma = \Lms_{\Gamma'}$. This proves \ref{it:Sforf1}.

Consider the topology on $\Gr(r,d)$ for which the proper closed sets are finite unions of intersections of sets of the form (these sets are precisely the maximal proper Schubert varieties)
\begin{equation*}
\bigl\{ V \in \Gr(r,d) \mid V \cap W \neq \{0\} \bigr\}, \quad W \in \Gr(d-r,d).
\end{equation*}
This topology is coarser than the Zariski topology hence is Noetherian. It follows that the closure of $\Lms_\Gamma$ in this topology has finitely many irreducible components.

The group $\Gamma$ permutes these irreducible components. Let $C$ be one of the components and let $\Gamma'$ be the stabilizer of $C$.
Then $\Gamma'$ is a subgroup of finite index in $\Gamma$. Remember from the argument for the first part that $\Lms_\Gamma = \Lms_{\Gamma'}$. 
Hence there exists $V \in \Lms_{\Gamma'} \cap C$. 
By the definition of $\Gamma'$, we know that the orbit $\Gamma' V \subset C$. 
Since $\Gamma$ acts strongly irreducibly on $\R^d$, so does $\Gamma'$. 
Hence by \cite[Lemma 4.2]{BenoistQuint}, $\Lms_{\Gamma'}$ is the closure of $\Gamma' V$ for the usual topology. Since $C$ is closed for the usual topology, we conclude that $\Lms_{\Gamma'} \subset C$. 
This shows that $\Lms_\Gamma$ has only one irreducible component.
Therefore, if $\Lms_\Gamma$ is not contained in any proper Schubert variety then it is not contained any finite union of proper Schubert varieties, finishing the proof of \ref{it:Sforf2}.

Finally, \ref{it:Sforf3} follows from \ref{it:Sforf2} and Lemma~\ref{lm:LimSI}.
\end{proof}

\subsection{Examples and non-examples.}
Now we shall see some examples of groups satisfying \eqref{eq:propertyS}. 

\begin{exemp}
\label{ex:SLdH}
Let $\mathbb{H}$ denote the usual real quaternion algebra. If the Zariski closure of $\Gamma \subset \GL(4d,\R)$ in $\GL(4d,\R)$ is $\SL(d,\mathbb{H})$, then $\Gamma$ satisfies \eqref{eq:propertyS}. 

To see this, recall that $\SL(d,\mathbb{H})$ can be defined in the following way. Fix a $\mathbb{H}$\dash{}structure on $\R^{4d}$, i.e. a morphism of algebra from $\mathbb{H}$ to $\End(\R^{4d})$. 
We say a $\R$-linear map $f \in \End(\R^{4d})$ is $\mathbb{H}$\dash{}linear if it commutes with $\mathbb{H}$. 
The group $\SL(d,\mathbb{H})$ is then the set of all elements $g \in \SL(4d,\R)$ that are $\mathbb{H}$\dash{}linear.
Using transvections one proves easily that $\Lms_G$ is the set of all $\mathbb{H}$-lines in $\R^{4d}$, i.e. $4$\dash{}dimensional $\R$-linear subspace which are preserved by the multiplication by $\mathbb{H}$.
The group $\SL(d,\mathbb{H})$ satisfy \eqref{eq:propertyS}, for otherwise\footnote{The author would like to thank Linxiao Chen for suggesting this argument.} there would be $W \in \Gr(4d - 4, 4d)$ such that $\mathbb{H} v \cap W \neq \{0\}$ for all $v \in \R^{4d} \setminus \{0\}$. Then the map $\mathbb{P}(\mathbb{H}) \times \mathbb{P}(W) \to \mathbb{P}(\R^{4d})$ defined by
\[(\R a, \R w) \mapsto \R aw \text{ for all } a \in \mathbb{H} \setminus \{0\} \text{ and } w \in W \setminus \{0\}\]
would be an differential map from a manifold of dimension $4d-2$ onto a manifold of dimension $4d-1$. This is impossible by Sard's theorem.
The property \eqref{eq:propertyS} of $\Gamma$ follows from that of $\SL(d,\mathbb{H})$.

Note also that $\SL(d,\mathbb{H})$ acts strongly irreducibly on $\R^{4d}$ since its Zariski closure is Zariski connected (because it is generated by transvections) and the commutant of $\SL(d,\mathbb{H})$ in $\End(\R^{4d})$ is precisely $\mathbb{H}$.
\end{exemp}

Next, we will prove that groups obtained from restriction of scalars from $\C$ to $\R$ satisfy \eqref{eq:propertyS}.

\begin{prop}
\label{pr:Cgroups}
Let $d \geq 2$ be an integer. Let $\mathbb{G} < \GL_d$ be a connected algebraic group over $\C$ for which the the standard representation $\C^d$ is irreducible. Let $\mathbb{G}_\R < \GL_{2d}$ denote the restriction of scalar of $\mathbb{G}$ to the ground field $\R$ and let $G < \GL_{2d}(\R)$ be the group of $\R$-points of $\mathbb{G}_\R$. Then the action of $G$ on $\R^{2d}$ is strongly irreducible, has proximal dimension $2$ and satisfies \eqref{eq:propertyS}.
\end{prop}

In view of Lemmata~\ref{lm:Sforfindex} and \ref{lm:SisZariski}, the conclusion also holds for any subsemigroup $\Gamma$ of $\GL_{2d}(\R)$ if the group of $\R$-points of the identity component of the Zariski closure of $\Gamma$ is such $G$.

\begin{proof}
Note that, as abstract groups, $\mathbb{G}$ and $G$ are isomorphic. We identify $\C^d$ with $\R^{2d}$ and view $\GL_d(\C)$ as a subgroup of $\GL_{2d}(\R)$. Then the underlying sets of $\mathbb{G}$ and of $G$ are the same.

Since $\mathbb{G}$ has a faithful irreducible representation, $\mathbb{G}$ is reductive. By \cite[\S 12.4.5]{Springer}, as an algebraic group, $\mathbb{G}_\R$ is isomorphic to $\mathbb{G} \times \mathbb{G}$. Hence $\mathbb{G}_\R$ is Zariski connected and reductive. In particular, $\R^{2d}$ is totally reducible as a linear representation of $G$ over $\R$.

We first show that the proximal dimension of $G$ is equal to $2$.
On the one hand, any limit of a converging sequence $(\lambda_n g_n)$ with $\lambda_n \in \R$ and $g_n \in \GL(d,\C)$ inside the space of $\R$-linear endomorphisms of $\C^d$ is actually $\C$-linear. Every nonzero $\C$-linear endomorphism has $\R$-rank at least $2$. Therefore the proximal dimension of $G$ is at least $2$. 
On the other hand, using the the theory of highest weight we can find a sequence of $g_n \in G$ such that $\norm{g_n}^{-1}g_n$ converges to a nonzero endomorphism of $\C$-rank equal to $1$. More precisely, let $\mathbb{T}$ be a maximal torus of $\mathbb{G}$. Choose a set of positive roots. Consider the decomposition of $\C^d$ into weight spaces with respect to $\mathbb{T}$. We know that there is a highest weight and the corresponding weight space has dimension is $1$. Take $\lambda \colon \C^* \to \mathbb{T}$ to be a multiplicative one-parameter subgroup inside the Weyl chamber defined by the system of positive roots. Define $g_n = \lambda(n)$, $n \geq 1$ and it is easy to see that $\norm{g_n}^{-1}g_n$ converges to a projection onto the highest weight space. 

In the meanwhile, we saw that the limit set $\Lms_G$ consists of complex lines. To finish the proof, we need a lemma.
\begin{lemm}
Let $W$ be a $\R$-linear subspace of $\C^d$ of real codimension $2$ and let $V \subset \C^d$ be a complex line $V = \C v$. There exists $g \in G$ such that $gV \cap W = \{0\}$.
\end{lemm}
\begin{proof}
Indeed, the intersection $W_0 = W \cap iW$ is a $\C$\dash{}linear subspace of $\C$\dash{}codimension $1$ or $2$. If $W_0$ has $\C$\dash{}codimension $1$, then $W = W_0$ and by the $\C$-irreducibility of $G$, there is $g \in G$ such that $gV \cap W = \{0\}$. Otherwise there exist two vectors $w_1, w_2 \in \C^d$ such that $\C^d = W_0 \oplus \C w_1 \oplus \C w_2$ and $W = W_0 \oplus \R w_1 \oplus \R w_2$. Assume for a contradiction that for all $g \in G$, $gV \cap W \neq \{0\}$. We claim that then the orbit $G V$ is contained in a finite union of $\C$-linear hyperplanes, which is impossible because $\mathbb{G}$ is Zariski connected and acts irreducibly on $\C^d$.

In order to prove the claim, we work in the $\C$-Zariski topology in the projective space $\mathbb{P}_\C(\C^d)$. Indeed, being the image of a morphism of varieties, the orbit $G V$ is a constructible set. Let $U$ denote the complement of $\mathbb{P}_\C(\C w_2 \oplus W_0)$ in $\mathbb{P}_\C(\C^d)$, which is an open set. The intersection $U \cap G V$ is again constructible. Consider the coordinate projection $p : U \to \C$, $p(x) = x_2$ where $x_2$ is the unique element in $\C$ such that $x \subset \C (w_1 + x_2 w_2) \oplus W_0$.
Being the image of a constructible set, $p(U \cap G V)$ is constructible. Moreover, since $g V$ intersect non-trivially $W$ for every $g \in G$, we have $p(U \cap G V) \subset \R$. But the only constructible subsets of $\C$ contained in $\R$ are the finite subsets. Hence $p(U \cap G V)$ is a finite set of real numbers, say $\{a_1, \dotsc a_N\}$. This means
\[G V \subset \mathbb{P}_\C( \C w_2 \oplus W_0) \cup \bigcup_{k=1}^N \mathbb{P}_\C( \C(w_1 + a_k w_2) \oplus W_0),\]
which concludes the proof of the claim and that of the lemma.
\end{proof}

Now we conclude the proposition from the lemma. Assume for a contradiction that the action of $G$ on $\R^{2d}$ is not irreducible. Then by the complete reducibility there is a $G$-invariant subspace $W$ of $\R^{2d}$ with dimension $\dim_\R(W) \leq \frac{2d}{2} \leq 2d - 2$.
Hence for a nonzero vector $v \in W$, we have $\C gv \cap W \neq \{0\}$ for all $g \in G$.
This contradicts the lemma. Therefore the action of $G$ on $\R^{2d}$ must be irreducible.
It is strongly irreducible because $\mathbb{G}_\R$ is Zariski connected. 
Finally, the property~\eqref{eq:propertyS} follows immediately form the lemma because every element in $\Lms_G$ is a complex line.
\end{proof}




We also have examples where \eqref{eq:propertyS} is not satisfied.

\begin{exemp}
Consider $\SO(1,d)$, $d > 6$ acting on $\wedge^2 \R^{1+d}$. Let $\Gamma < \GL(\wedge^2 \R^{1+d})$ be the corresponding subgroup. The action is strongly irreducible. The proximal dimension is $d - 1$. However, there are more than one disjoint compact $\Gamma$-invariant subsets in $\Gr(d-1, \wedge^2 \R^{1+d})$ , see \cite[Remark 4.4]{BenoistQuint}. 
In view of Lemma~\ref{lm:LimSI}, $\Gamma$ does not satisfy \eqref{eq:propertyS}.
\end{exemp}

%% file: grass.tex
\section{Random walk on the Grassmannian}
\label{sc:grass}

Given a Borel probability measure $\mu$ on $\GL_d(\R)$, it induces a random walk on each of the Grassmannian varieties $\Gr(k,d)$, $1\leq k \leq d -1$. Here we are interested in the random walk on $\Gr(r,d)$, where $r = r_\Gamma$ is the proximal dimension of $\Gamma = \Gamma_\mu$, the closed subsemigroup generated by $\Supp(\mu)$. 
The principal goal of this section is a large deviation estimate in Proposition~\ref{pr:RWonGr} for groups satisfying \eqref{eq:propertyS}. This can be interpreted as a regularity result for the $\mu$-stationary measure on $\Gr(r,d)$, as shown in Corollary~\ref{cr:regularity}.

For $k = 1 ,\dotsc,d$, we endow the Grassmannian $\Gr(k,d)$ with the distance $\dH$ defined by
\[\forall V, V' \in \Gr(k,d),\quad \dH(V,V') = \max_{v \in V : \norm{v} = 1} d(v,V').\]
Equivalently $\dH(V,V')$ is the Hausdorff distance between the closed unit ball in $V$ and that in $V'$.

Recall that the Euclidean norm on $\R^d$ induces an Euclidean norm on $\wedge^k\R^d$ for each $k = 1 ,\dotsc, d$. For linear subspaces $V, W \subset \R^d$ of dimension respectively $t$ and $s$, we define
\[\dang(V,W) = \frac{\norm{v_1 \wedge \dotsm \wedge v_t \wedge w_1 \wedge \dotsm \wedge w_s}}{\norm{v_1 \wedge \dotsm \wedge v_t} \norm{w_1 \wedge \dotsm \wedge w_s }}\]
where $(v_1,\dotsc,v_t)$ is a basis of $V$ and $(w_1,\dotsc,w_s)$ a basis of $W$. Restricted to the projective space $\mathbb{P}(\R^d)$, $\dang(\bullet,\bullet)$ and $\dH(\bullet,\bullet)$ coincides. In other cases, $\dang$ is not a distance. For instance, $\dang(V,W) = 0$ if and only if $V$ and $W$ have nontrivial intersection. Thus $\dang(V,W)$ measures how far $V$ is away from the Schubert variety $\{V \in \Gr(t,d) \mid V \cap W \neq \{0\} \}$.

For $g \in \GL_d(\R)$, consider its Cartan decomposition $g = k \diag(\sigma_1(g),\dotsc,\sigma_d(g))l$ where $\sigma_1(g) \geq \dotsc \geq \sigma_d(g) > 0$ are the singular values of $g$. Define
\[V^+_g = k\Span(e_1,\dotsc,e_r) \text{ and } V^-_g = l^{-1}\Span(e_{r+1},\dotsc, e_d)\]
where $(e_1,\dotsc, e_d)$ is the standard basis of $\R^d$. In the case where $r = 1$, $V^+_g$ corresponds to the notation $\theta(g)$ in \cite{BFLM} and $V^-_g$ corresponds to $H(g)$.

Note that when $\sigma_r(g) = \sigma_{r+1}(g)$, the Cartan decomposition of $g$ is not unique and $V^+_g$ and $V^-_g$ are not uniquely defined. However, this inconvenience does not matter for our purpose. In this case, simply choose an arbitrary Cartan decomposition of $g$ and $V^+_g$ and $V^-_g$ refer to the associated subspaces. Obviously, this choice can be made in a measurable manner.

For each $k = 1 ,\dotsc, d$, denote by $\lambda_{k,\mu}$ the $k$-th Lyapunov exponent associated to the random walk defined by $\mu$. Recall that it can be defined as
\[\lambda_{k,\mu} = \lim_{n \to +\infty} \frac{1}{n}\int \sigma_k(g) \dd \mu^{*n}(g)\]
by the law of large numbers due to Furstenberg~\cite{Furstenberg1963}. A fundamental result of Guivarc'h-Raugi~\cite{GuivarchRaugi} states that if $\Gamma$ acts strongly irreducibly on $\R^d$, then
\[\lambda_{1,\mu} = \dotsb = \lambda_{r,\mu} > \lambda_{r + 1,\mu}.\]

\begin{prop}\label{pr:RWonGr}
Let $\mu$, $\Gamma$ and $r$ be as above. Assume that $\mu$ has finite exponential moment, $\Gamma$ acts strongly irreducibly on $\R^d$ and satisfies \eqref{eq:propertyS}.
Then for any $\omega > 0$, there is $c > 0$ and $l_0 \geq 1$ such that for all $n \geq l \geq l_0$, the following holds.
\begin{enumerate}
\item \label{it:danggVW} For any $V \in \Gr(r,d)$ and any $W \in \Gr(d-r,d)$,
\[\mu^{*n} \{ g \in \Gamma \mid \dang(g V,W) \leq e^{-\omega l} \} \leq e^{-cl}.\]
\item \label{it:danggVV+} For any $V \in \Gr(r,d)$,
\[\mu^{*n} \{ g \in \Gamma \mid \dH(g V,V^+_g) \leq e^{-(\lambda_{1,\mu} - \lambda_{r+1,\mu}  -\omega) n} \} \geq 1 - e^{-c n}.\]
\item \label{it:dangV+W} For any $W \in \Gr(d-r,d)$,
\[\mu^{*n} \{ g \in \Gamma \mid \dang(V^+_g,W) \leq e^{-\omega l} \} \leq e^{-cl}.\]
\end{enumerate}
\end{prop}
      

\subsection{Proof of Proposition~\ref{pr:RWonGr}}
The main tool we need is the following large deviation estimates about random matrix products. The first of such result, due to Le Page~\cite{LePage}, is the item~\ref{it:LargeDnc} and the case $k=1$ of the item~\ref{it:LargeDsv} assuming $\Gamma$ strongly irreducible and proximal. Bougerol~\cite[Theorem V.6.2]{BougerolLacroix} extended these to the case where $\Gamma$ is only assumed to be strongly irreducible. The item~\ref{it:LargeDsv} for $k = 1$ appeared in \cite[Theorem 3.4]{Breuillard}. It can be obtained by applying Bougerol's result to each of the irreducible subrepresentations and then using an estimate due to Aoun~\cite[Lemma 2.4.43]{Aoun_these} about return time to finite index subgroups. The full generality of the item~\ref{it:LargeDsv} follows then by applying \cite[Theorem 3.4]{Breuillard} to each of the representations $\wedge^k \R^d$. The item~\ref{it:LargeDmc} is essentially proved in \cite[Theorem 4.4, Lemma 4.5]{BFLM} although there it is formulated differently. In its formulation below, the item~\ref{it:LargeDmc} is \cite[Lemma 14.11]{BenoistQuint}.
\begin{thm}[Large deviation estimates]\label{thm:LargeD}
Let $\mu$ be a Borel probability measure on $\GL_d(\R)$ with finite exponential moment. 
For any $\omega > 0$, there is $c > 0$, $l_0 > 0$ such that for all $n \geq l \geq l_0$, the following holds.
\begin{enumerate}
\item \label{it:LargeDsv} If the action of $\Gamma$ on $\R^d$ is completely reducible, then for all $k = 1, \dotsc, d$,
\[\mu^{*n} \bigl\{ g \in \Gamma \mid \abs{ \frac{1}{n}\log \sigma_k(g) - \lambda_{k,\mu}} \geq \omega \bigr\} \leq e^{-cn}.\]
\item \label{it:LargeDnc} If the action of $\Gamma$ on $\R^d$ is strongly irreducible, then for all nonzero vectors $x \in \R^d$,
\[\mu^{*n} \bigl\{ g \in \Gamma \mid \abs{ \frac{1}{n}\log \frac{\norm{gx}}{\norm{x}} - \lambda_{1,\mu}} \geq \omega \bigr\} \leq e^{-cn}.\]
\item \label{it:LargeDmc} If the action of $\Gamma$ on $\R^d$ is strongly irreducible and proximal then for all nonzero vectors $x \in \R^d$ and all nonzero linear forms $f \in (\R^d)^*$,
\[\mu^{*n} \bigl\{g \in \Gamma \mid \abs{f(gx)}\leq e^{-\omega l}\norm{gx}\,\norm{f}\bigr\} \leq e^{-cl}.\]
\end{enumerate}
\end{thm}



The main idea in proving Proposition~\ref{pr:RWonGr} is to apply Theorem~\ref{thm:LargeD}\ref{it:LargeDmc} to the representation $\wedge^r\R^d$. The issue is : although the representation $\wedge^r\R^d$ is always proximal, it is not irreducible in general. Thus, we need a decomposition of $\wedge^r\R^d$, which we describe below. This decomposition is first used by Bougerol in~\cite[Theorem V.6.2]{BougerolLacroix}, while Benoist-Quint~\cite[Lemma 4.36]{BenoistQuint} and Breuillard~\cite[Lemma 3.2]{Breuillard} provided more detailed description.

Assume that $\Gamma$ acts strongly irreducibly on $\R^d$. We have  a direct sum decomposition into $\Gamma$-invariant subspaces
\[\wedge^r \R^d = \Lambda_+ \oplus \Lambda_0\]
where
\[\Lambda_+ = \sum_{\pi \in \Pi_\Gamma} \im (\wedge^r \pi)\]
and 
\[\Lambda_0 = \bigcap_{\pi \in \Pi_\Gamma} \ker (\wedge^r \pi).\]

The action of $\Gamma$ on $\Lambda_+$ is strongly irreducible and proximal. Moreover,
\[\forall g \in \Gamma,\quad \norm{g}^r \ll_\Gamma \norm{(\wedge^r g)_{\mid \Lambda_+}} \leq \norm{g}^r.\]
Consequently, the top Lyapunov exponent associated to the random walk on $\Lambda_+$ is $\lambda_{1,\Lambda_+} = r\lambda_{1,\mu}$.
The action of $\Gamma$ on $\Lambda_0$ is totally reducible and the corresponding top Lyapunov exponent satisfies
\begin{equation}
\lambda_{1,\Lambda_0} \leq (r-1) \lambda_{1,\mu} + \lambda_{r+1,\mu}.
\end{equation}

\begin{proof}[Proof of Proposition~\ref{pr:RWonGr}]
For $V \in \Gr(r,d)$, let $\vb \in \wedge^r \R^d$ be the wedge product of an orthonormal basis of $V$. The vector $\vb$ is, up to a sign, uniquely determined by $V$. Similarly, define $\wb \in \wedge^{d-r}\R^d$ for $W \in \Gr(d-r,d)$. We have
\[\dang(gV,W) = \frac{\norm{(\wedge^r g) \vb \wedge \wb}}{\norm{(\wedge^r g) \vb}}.\]

Write $\vb$ as $\vb = \vb_+ + \vb_0$ with $\vb_+ \in \Lambda_+$ and $\vb_0 \in \Lambda_0$. Write also $f_\wb \colon \Lambda_+ \to \wedge^d \R^d$ for the linear form $\ub \mapsto \ub \wedge \wb$ restricted to $\Lambda_+$.
We claim that under the assumption of \eqref{eq:propertyS}, there exists a constant $c > 0$ depending only on $\Gamma$ such that $\norm{\vb_+} \geq c$ and $\norm{f_\wb} \geq c$.

Indeed, by the definition of $\Lambda_0$, $\vb_+ = 0$ if and only if $\forall \pi \in \Pi_\Gamma$, $\vb \in \ker (\wedge^r \pi)$, if and only if $\forall \pi \in \Pi_\Gamma$, $V \cap \ker \pi \neq \{0\}$.
Thus, \eqref{eq:kerSI} implies $\vb_+ \neq 0$. Observe that the map $\Gr(r,d) \to \R$, $V \mapsto \norm{\vb_+}$ is well-defined and continuous. The compactness of $\Gr(r,d)$ then implies 
\[\inf_{V \in \Gr(r,d)} \norm{\vb_+} > 0.\]
Similaily, by the definition of $\Lambda_+$, $f_\wb = 0$ if and only if $\forall \pi \in \Pi_\Gamma$, $f_\wb(\im(\wedge^r \pi)) = \{0\}$, if and only if $\forall \pi \in \Pi_\Gamma$, $\im \pi \cap W \neq \{0\}$.
Thus, \eqref{eq:imSI} implies $f_\wb \neq 0$. Again from the compactness of $\Gr(d-r,d)$, we conclude 
\[\inf_{W \in \Gr(d-r,d)} \norm{f_\wb} > 0.\]

For any $g \in \Gamma$, we have 
\begin{equation}
\label{eq:rgvb}
(\wedge^r g) \vb = (\wedge^r g) \vb_+ + (\wedge^r g) \vb_0
\end{equation}
and
\begin{equation}
\label{eq:rgvbwb}
 (\wedge^r g) \vb \wedge \wb = (\wedge^r g) \vb_+ \wedge \wb + (\wedge^r g) \vb_0 \wedge \wb.
\end{equation}
We can bound from above the second terms in both right-hand sides:
\begin{equation}
\label{eq:rgvb0wb}
\norm{(\wedge^r g) \vb_0 \wedge \wb} \leq \norm{(\wedge^r g) \vb_0 } \ll_\Gamma \norm{(\wedge^r g)_{\mid \Lambda_0}}.
\end{equation}

Let $0 < \omega < \frac{\lambda_{1,\mu} - \lambda_{r+1,\mu}}{3}$. 
In the argument below, $c > 0$ and $l_0 > 0$ will denote the constants given by Theorem~\ref{thm:LargeD} when applied to random walks induced by $\mu$ on $\R^d$ and on $\Lambda_+$.
Let $n \geq l \geq l_0$ as in the statement of Proposition~\ref{pr:RWonGr}.
Let $g$ denote a random variable distributed according to $\mu^{*n}$.

As discussed above, the action of $\Gamma$ on $\Lambda_+$ is strongly irreducible and proximal. Thus we can apply Theorem~\ref{thm:LargeD}\ref{it:LargeDnc} and Theorem~\ref{thm:LargeD}\ref{it:LargeDmc} to the random walk on $\Lambda_+$ induced by $\mu$. We obtain that with probability at least $1 - e^{-cn}$,
\[\norm{(\wedge^r g)\vb_+} \geq e^{(r\lambda_{1,\mu}-\omega)n}\norm{\vb_+} \gg_\Gamma e^{(r\lambda_{1,\mu}-\omega)n},\]
and with probability at least $1 - e^{-cl}$,
\[\norm{(\wedge^r g) \vb_+ \wedge \wb} \geq e^{-\omega l} \norm{f_\wb} \norm{(\wedge^r g)\vb_+} \gg_\Gamma e^{-\omega l}\norm{(\wedge^r g)\vb_+}.\]
The action of $\Gamma$ on $\Lambda_0$ is totally reducible. Applying Theorem~\ref{thm:LargeD}\ref{it:LargeDsv} with $k=1$ to the associated random walk, we have, with probability at least $1 - e^{cn}$,
\[\norm{(\wedge^r g)_{\mid \Lambda_0}} \leq e^{((r-1)\lambda_{1,\mu} + \lambda_{r+1,\mu} + \omega)n}.\]

When all these happen we will have $\norm{(\wedge^r g) \vb_0 } \leq \norm{(\wedge^r g) \vb_+}$ and $\norm{(\wedge^r g) \vb_0 \wedge \wb} \leq \frac{1}{2} \norm{(\wedge^r g) \vb_+ \wedge \wb}$. These inequalities combined with \eqref{eq:rgvb}, \eqref{eq:rgvbwb} and \eqref{eq:rgvb0wb} yield $\norm{(\wedge^r g) \vb} \leq 2 \norm{(\wedge^r g) \vb_+}$ and  $\norm{(\wedge^r g) \vb \wedge \wb} \geq \frac{1}{2}\norm{(\wedge^r g) \vb_+ \wedge \wb}$.

Hence, with probability greater than $1 - e^{-cl}$,
\[\dang(gV,W) \gg  \frac{\norm{(\wedge^r g) \vb_+ \wedge \wb}}{\norm{(\wedge^r g) \vb_+}} \gg_\Gamma e^{-\omega l}.\]
This finishes the proof of \ref{it:danggVW}. 

In order to prove \ref{it:danggVV+}, we use the following fact.
\begin{lemm}
For any $g \in \GL(\R^d)$ and any $V \in \Gr(r,d)$, 
\begin{equation}
\label{eq:gVV+close}
\dH(gV,V^+_g) \leq \frac{r\sigma_1(g)^{r-1}\sigma_{r+1}(g)}{\norm{(\wedge^r g)\vb}}.
\end{equation}
\end{lemm}

\begin{proof}
Let $(v_1,\dotsc,v_r)$ be an orthonormal basis and $\vb = v_1 \wedge \dotsm \wedge v_r$.

Write each $v_i$ as $v_i = v^+_i + v^-_i$ with $v^+_i \in (V^-_g)^\perp$ and $v^-_i \in V^-_g$. We have
\[d(gv_i, V^+_g) = \norm{gv^-_i} \leq \sigma_{r+1}(g)\norm{v^-_i} \leq \sigma_{r+1}(g).\]
Now let $u \in gV$, there is $(\alpha_i) \in \R^d$ such that
\[u = \sum_{i=1}^r \alpha_i gv_i.\]
Taking the wedge product of $u$ with all the $gv_i$ except one, we obtain $\forall i = 1,\dotsc,r$,
\[\abs{\alpha_i} \norm{(\wedge^r g)\vb} \leq \norm{g}^{r-1}\norm{u}.\]
From the above follows that $\forall u \in gV$,
\[d(u,V^+_g) \leq \frac{r\sigma_1(g)^{r-1}\sigma_{r+1}(g)}{\norm{(\wedge^r g)\vb}} \norm{u}.\]
This proves the lemma.
\end{proof}

Now using large deviation estimates, we can control each of the terms appearing in the right hand side of \eqref{eq:gVV+close}. By Theorem~\ref{thm:LargeD}\ref{it:LargeDsv},
\[ \sigma_1(g)^{r-1}\sigma_{r+1}(g) \leq e^{((r-1)\lambda_{1,\mu} + \lambda_{r+1,\mu} +\omega)n}\]
with probability at least $1 - e^{-cn}$.
From the first part of this proof,
\begin{equation}
\norm{(\wedge^r g)\vb} \geq \norm{(\wedge^r g)\vb_+} - \norm{(\wedge^r g)_{\mid \Lambda_0}} \geq e^{(r\lambda_{1,\mu} - \omega)n}
\end{equation}
with probability at least $1 - e^{-cn}$. Thus \ref{it:danggVV+} follows from these inequalities and the lemma.

Finally, \ref{it:dangV+W} follows immediately from \ref{it:danggVW}, \ref{it:danggVV+} and following triangular inequality. For all $V,V' \in \Gr(r,d)$ and all $W \in \Gr(d-r,d)$,
\[\abs{\dang(V,W) - \dang(V',W)} \leq 2r \dH(V,V').\]
The proof of this inequality is straightforward and omitted.
\end{proof}

\subsection{Stationary measure}
The remainder of this section is irrelevant to the main result (Theorem~\ref{thm:main}) of this article. We will provide an interpretation of Proposition~\ref{pr:RWonGr} in terms of stationary measure in Corollary~\ref{cr:regularity}.
In the meanwhile, we show in Proposition~\ref{pr:uniquenu} the uniqueness of stationary measure on the Grassmannian $\Gr(r_\Gamma, d)$ under the assumption~\eqref{eq:propertyS}, just like for proximal groups.

We use the notation introduced in the beginning of this section. Recall that a Borel probability measure $\nu$ on $\Gr(r,d)$ is said to be $\mu$-stationary if $\mu * \nu = \nu$. 

\begin{prop}
\label{pr:uniquenu}
Assume that $\Gamma$ acts strongly irreducibly on $\R^d$ and satisfies \eqref{eq:propertyS}. Then there is a unique $\mu$\dash{}stationary Borel probability measure on $\Gr(r,d)$.
\end{prop}
For the proximal case, i.e. $r=1$, this is a classical result due to Furstenberg~\cite{Furstenberg1973}. 

We have seen in Lemma~\ref{lm:LimSI} that, under the assumption~\eqref{eq:propertyS}, the limit set $\Lms_\Gamma$ is the unique minimal $\Gamma$-invariant subset of $\Gr(r,d)$. By \cite[Remark 10.5]{BenoistQuint}, it follows that there is a unique $\mu$\dash{}stationary probability measure supported on $\Lms_\Gamma$. The proposition asserts that it is indeed the only one on $\Gr(r,d)$. The proof is similar to that of Furstenberg's result (cf. \cite[Proposition 4.7]{BenoistQuint}). In particular, we make use of the limit measures and the boundary maps constructed by Furstenberg~\cite{Furstenberg1963}. 

\begin{proof}
Let $(B,\mathcal{B},\beta)$ denote the probability space with $B = \Gamma^{\N^*}$, $\mathcal{B}$ being the product $\sigma$-algebra of the Borel $\sigma$-algebra on $\Gamma$ and $\beta = \mu^{\otimes {\N^*}}$ being the product measure. 
Let $\mathcal{P}(\Gr(r,d))$ denote the space of Borel probability measures on $\Gr(r,d)$ endowed with the weak-$*$ topology. Let $\nu \in \mathcal{P}(\Gr(r,d))$ be a $\mu$-stationary measure. 

Recall the definition and properties of the limit measures, cf. \cite[Lemma 2.17, Lemma 2.19]{BenoistQuint}. 
There exists a Borel map $b \mapsto \nu_b$ from $B$ to $\mathcal{P}(\Gr(r,d))$ such that for $\beta$-almost any $b = (b_n)_{n \geq 1}$ in $B$, one has $(b_1 \dotsm b_n)_*\nu \to \nu_b$ as $n \to +\infty$.
Moreover,
\begin{equation}
\label{eq:nuInub}
\nu = \int_B\nu_b \dd \beta(b).    
\end{equation}

Recall the definition of Furstenberg boundary map. There exists a Borel map $\xi \colon B \to \Gr(r,d)$ such that for $\beta$-almost any $b = (b_n)_{n \geq 1}$ in $B$, every nonzero accumulation point $f \in \End(\R^d)$ of a sequence $\lambda_n b_1 \dotsm b_n$ with $\lambda_n \in \R$ has rank $r$ and admits $\xi(b)$ as its image. 

We claim that for $\beta$-almost $b \in B$, the limit measure $\nu_b$ is the Dirac mass at $\xi(b)$. In view of \eqref{eq:nuInub}, the uniqueness of $\nu$ follows. 

Indeed, for $b = (b_n)_{n \geq 1} \in B$, let $\pi_b \in \End(\R^d)$ be an accumulation point of the sequence $\lambda_n b_1 \dotsm b_n$ where $\lambda_n = \norm{b_1 \dotsm b_n}^{-1}$. For $\beta$-almost all $b$, $\pi_b$ has rank $r$ and $\im(\pi_b) = \xi(b)$. 
Let $W_b = \ker \pi_b$.
By \cite[Lemma 10.16]{BenoistQuint}, there is a constant $c > 0$ depending only on $\Gamma$ such that $\sigma_r(\pi_b) \geq c \sigma_1(\pi_b)= c$. Thus from the Cartan decomposition of $\pi_b$, we see that
\begin{equation}
\label{eq:pibvdang}
\forall v \in \R^d \setminus \{0\},\; \norm{\pi_b v} \geq c \dang(\bar{v},W_b) \norm{v}.
\end{equation}

Define $\epsilon_n = \norm{\lambda_n{b_1 \dotsm b_n} - \pi_b }^{1/2}$
and $\rho_n = (\epsilon_n + \epsilon_n^2)/c$ so that both $\epsilon_n \to 0$ and $\rho_n \to 0$ along a subsequence.
For $V \in \Gr(r,d)$ such that $\dang(V,W_b) \geq \rho_n$, we have, $\forall v \in V$,
\[\norm{\lambda_n{b_1 \dotsm b_n}v - \pi_b v }\leq \epsilon_n^2 \norm{v}\]
and by~\eqref{eq:pibvdang}
\[\norm{\pi_b v} \geq c \rho_n \norm{v}\]
Hence 
\[\norm{\lambda_n{b_1 \dotsm b_n} v} \geq (c\rho_n - \epsilon_n^2) \norm{v} = \epsilon_n \norm{v} \]
and 
\[d \bigl({b_1 \dotsm b_n}v, \xi(b)\bigr) \leq \epsilon_n \norm{{b_1 \dotsm b_n}v}.\]
Therefore,
\[\dH \bigl({b_1 \dotsm b_n}V, \xi(b)\bigr) \leq \epsilon_n.\]
What we have shown is
\begin{multline}
\label{eq:limitnuDirac}
((b_1 \dotsm b_n)_*\nu )\bigl\{V \in \Gr(r,d) \mid \dH(V,\xi(b)) \leq \epsilon_n \bigr\} \\
\geq \nu\{V \in \Gr(r,d) \mid \dang(V,W_b) \geq \rho_n\}.
\end{multline}

By Lemma~\ref{lm:nuSch=0} below, we have
\[\nu \bigl\{ V \in \Gr(r,d) \mid  \dang(V,W_b) = 0 \bigr\} = 0.\]
Hence 
\[\nu \bigl\{ V \in \Gr(r,d) \mid  \dang(V,W_b) \leq \rho \bigr\} \to 0 \quad \text{as } \rho \to 0.\]

Let $n$ goes to $+\infty$ in \eqref{eq:limitnuDirac}, we obtain $\forall \epsilon > 0$
\[\nu_b \bigl\{V \in \Gr(r,d) \mid \dH(V,\xi(b)) \leq \epsilon \bigr\} = 1,\]
showing that $\nu_b$ is the Dirac mass at $\xi(b)$.
\end{proof}

Note that the claim implies that this $\mu$-stationary measure on $\Gr(r,d)$ is $\mu$-proximal (for the definition of this notion see \cite[Section 2.7]{BenoistQuint}).

In the proof above we used the following lemma.
\begin{lemm}
\label{lm:nuSch=0}
Assume that $\Gamma$ acts strongly irreducibly on $\R^d$ and satisfies \eqref{eq:propertyS}.
Let $\nu$ be a $\mu$\dash{}stationary probability measure on $\Gr(r,d)$.
Then for any $W \in \Gr(d - r,d)$,
\[\nu \bigl\{ V \in \Gr(r,d) \mid V \cap W \neq \{0\} \bigr\} = 0.\]
\end{lemm}

\begin{proof}
We work again in the topology introduced in the proof of Lemma~\ref{lm:Sforfindex}. As for Zariski topology, we can define the dimension of an irreducible closed set as the maximal length of increasing chain of irreducible closed subsets. 

Assume for a contradiction that $\nu$ gives positive mass to a proper closed set.
Let $M$ denote the set of irreducible closed sets of minimal dimension and with maximal $\nu$-measure among irreducible closed sets of minimal dimension. Using an argument of Furstenberg (cf. \cite[Lemma 4.6(b)]{BenoistQuint}), one can prove that $M$ is finite and for all $g \in \Supp(\mu)$ and all $F \in M$, $g^{-1}F \in M$.
From this we deduce that $g^{-1}$ permutes the finite set $M$ and hence so does $g$. 
Thus, $\sg{\Gamma}$, the subgroup generated by $\Gamma$, acts on $M$.
Then, on the one hand, $\sg{\Gamma}$ preserve the finite union $\bigcup_{F \in M} F$.
On the other hand, since $\sg{\Gamma}$ has the same Zariski closure as $\Gamma$, it has property~\eqref{eq:propertyS} by Lemma~\ref{lm:SisZariski}. We get a contradiction by Lemma~\ref{lm:Sforfindex}\ref{it:Sforf3} applied to $\sg{\Gamma}$.
\end{proof}

The following is a Hölder-regularity result for the stationary measure. It is a consequence of Proposition~\ref{pr:RWonGr}.

\begin{coro}[to Proposition~\ref{pr:RWonGr}]
\label{cr:regularity}
Under the same assumptions as in Proposition~\ref{pr:RWonGr}. Let $\nu$ be the $\mu$-stationary Borel probability measure on $\Gr(r,d)$. There exists $\kappa > 0$ such that for all $W \in \Gr(d-r,d)$ and all $\rho > 0$,
\[\nu\{V \in \Gr(r,d) \mid \dang(V,W) \leq \rho \} \leq \rho^\kappa.\]
Moreover,
\begin{equation}
\label{eq:regonGr}
\sup_{W \in \Gr(d-r,d)}  \int_{\Gr(r,d)} \dang(V,W)^{-\kappa} \dd \nu(V) < + \infty. 
\end{equation}
\end{coro}
For $r = 1$, this is due to Le Page~\cite{LePage} and it is a direct consequence of Theorem~\ref{thm:LargeD}\ref{it:LargeDnc}. For $r > 1$, one checks easily that \eqref{eq:regonGr} is stronger than \cite[Theorem 14.5]{BenoistQuint} which hold without the assumption \eqref{eq:propertyS}. On the other hand, \eqref{eq:regonGr} clearly fails if the limit set of $\Gamma$ is contained in a proper Schubert variety.
\begin{proof}
The proof follows the same pattern as in the proof of~\cite[Theorem 14.1]{BenoistQuint}.
\end{proof}

%% file: check.tex
\section{Random walk on the Torus}
\label{sc:check}
The goal of this section is to explain the proof of Theorem~\ref{thm:main}, which is merely an adaptation of the proof in the Bourgain-Furman-Lindenstrauss-Mozes paper~\cite{BFLM}. We will indicate where the proximality assumption {\upshape (P)} is used in~\cite{BFLM} and how to adapt under the assumption~\eqref{eq:propertyS}.

Throughout this section, $\mu$ denotes a probability measure on $\SL_d(\Z)$ with finite exponential moment. Let $\Gamma$ denote the subsemigroup generated by $\Supp(\mu)$. Let $r = r_\Gamma$ denote the proximal dimension of $\Gamma$.
Assume that $\Gamma$ acts strongly irreducibly on $\R^d$ and satisfies \eqref{eq:propertyS} and that $r$ divides $d$. Note that necessarily $r < d$ because the proximal dimension of a subsemigroup of $\SL_d(\Z)$ is equal to $d$ if and only if the subsemigroup is finite and a finite group never acts strongly irreducibly on $\R^d$ for $d \geq 2$. Hence $d/r$ is an integer larger or equal than $2$.

Let $\nu_0$ be a Borel probability measure on $\mathbb{T}^d$ and write $\nu_n = \mu^{*n} * \nu_0$ for $n \in \N$. 
For $t > 0$ and integer $n \geq 0$, let $A_{t,n}$ be the set of large Fourier coefficients of $\nu_n$,
\[A_{t,n} = \{a \in \Z^d \mid \abs{\hat\nu_{n}(a)} \geq t\}. \]

For bounded subset $A \subset \R^d$ and $M > 0$, we define $\Ncov(A,M)$ as the minimal integer $N$ such that $A$ can be covered by $N$ balls of radius $M$. 

Recall very roughly the outline of the proof of Theorem~\ref{thm:BFLM} in~\cite{BFLM}. It consists of two phases. 
In Phase I (Theorem 6.1 in~\cite{BFLM}), one proves that if $A_{t_0,n_0}$ contains a nonzero element $a_0$ for some $t_0 \in {(0,1/2)}$ and some large $n_0$, then there is $n_1 \leq n_0$ such that 
\[\Ncov(A_{t_1,n_1} \cap B(0,N), M) \geq t_1 \bigl(\frac{N}{M}\bigr)^d\]
where $t_1 = t_0^{O(1)}$ and $N$, $M$ are quantities that can be bounded in certain range. 
Then by a harmonic analytic lemma (Proposition 7.5 in~\cite{BFLM}), one deduces that $\nu_{n_1}$ is granulated: there exists an $\frac{1}{M}$-separated set $X \subset \mathbb{T}^d$ such that
\[\nu_{n_1}\Bigl(\bigcup_{x \in X} B(x,\frac{1}{N})\Bigr) \geq t_2\]
where $t_2 = t_1^{O(1)}$.
In Phase II (Section 7 in~\cite{BFLM}), one shows that the granulation in $\nu_{n_2}$ become stronger as $n_2$ gets smaller and the grains (the set $X$ as above) are close to rational points of bounded heights. This allows to conclude.

Now we will indicate in each step, where the assumption {\upshape (P)} is used and what needs to be said if we only have assumption~\eqref{eq:propertyS} instead of {\upshape (P)}.

\subsection{Initial dimension of large Fourier coefficients.}
The first step in Phase I is to obtain an initial rough dimension of the large Fourier coefficients. This is Proposition 6.2 in~\cite{BFLM}. Here the proximality is used in the form of a non-concentration estimate of the associated random walk on $\mathbb{P}(\R^d)$:
given $\omega > 0$, there exists $c > 0$ such that for $m$ large enough and for all lines $x,y \in \mathbb{P}(\R^d)$,
\[\mu^{*m} \bigl\{g \in \Gamma \mid \dang(g x, y) \leq e^{-\omega m} \bigr\} \leq e^{-cm}.\]

For a group with property \eqref{eq:propertyS}, this non-concentration estimate is an immediate consequence of Proposiition~\ref{pr:RWonGr}. Thus Proposition 6.2 in~\cite{BFLM} hold under our assumption.

\subsection{Bootstrap of large scale dimension.}
Here we check that Proposition 6.3 in \cite{BFLM} holds under the assumption~\eqref{eq:propertyS}. This proposition allows to improve the rough dimension of the set of large Fourier coefficients $A_{t,n}$ by paying the price of making $t$ and $n$ smaller. 
This is where a discretized projection theorem is used. So this part presented the biggest obstacle in relaxing assumption {\upshape (P)} in Theorem~\ref{thm:BFLM}. For this reason we will give a detailed proof here.
First let us recall the statement.
\begin{prop}
\label{pr:prop6.3}
Given $\alphaini > 0$ and $\alphahigh < d$, there exists constants $\alphainc, C >0$ depending on $\alphaini$, $\alphahigh$ and $\mu$ such that if for some $0 < t < \frac{1}{2}$, $1 \leq M \leq N$ with
\[\log\frac{N}{M} > C \log\frac{1}{t} \text{ and } n \geq C\log\frac{N}{M}\]
it holds that 
\[\Ncov(A_{t,n} \cap B(0,N),M) \geq \Bigl(\frac{N}{M}\Bigr)^{\alpha} \text{ for some $\alphaini \leq \alpha \leq \alphahigh$},\]
then there is a new radius $N'> 0$, a new scale $M' \geq M$ and an integer $m \leq C \log\frac{N}{M}$ such that
\begin{equation}
\label{eq:propN'M'}
N' \leq N\Bigl(\frac{N}{M}\Bigr)^C \text{ and } \log\frac{N'}{M'} \geq \frac{1}{2}\log\frac{N}{M}
\end{equation}
and  
\[\Ncov(A_{t',n'} \cap B(0,N'),M') \geq  \Bigl(\frac{N'}{M'}\Bigr)^{\alpha + \alphainc}\]
for $t' = 2^{-10d}t^{2d}$ and $n' = n - m$.
\end{prop}

Let us state the discretized projection theorem that is needed here. For a subspace $V \in \Gr(r,d)$, let $\pi_V \colon \R^d \to V$ denote the orthogonal projection onto $V$.
\begin{thm}[{\cite[Theorem 1]{He_proj}}]
\label{thm:proj}
Let $0 < r <  d$ be integers. Given $0 < \alpha < d$, $\kappa > 0$ there exists $\tau > 0$ such that the following holds for $\delta > 0$ sufficiently small. Let $E \subset B(0,1)$ be a bounded subset of $\R^d$. Let $\eta$ be a probability measure on $\Gr(r,d)$. Assume that
\begin{equation}
\Ncov(E,\delta) \geq \delta^{-\alpha + \tau};
\end{equation}
\begin{equation}
\forall \rho \geq \delta,\, \max_{x \in \R^d} \Ncov(E \cap B(x,\rho),\delta) \leq \delta^{-\tau}\rho^\kappa\Ncov(E,\delta);
\end{equation}
\begin{equation}
\forall \rho \in {[\delta, \delta^\tau]},\, \forall W \in \Gr(d-r,d),\, \eta\{ V \in \Gr(r,d) \mid \dang(V,W)\} \leq \rho^\kappa.
\end{equation}
Then there exists $\mathcal{D} \subset \Gr(r,d)$ such that $\eta(\mathcal{D}) \geq 1 - \delta^\tau$ and
\[\Ncov(\pi_V(E'),\delta) \geq \delta^{-\frac{r}{d}\alpha - \tau}\]
whenever $V \in \mathcal{D}$ and $E' \subset E$ is a subset such that $\Ncov(E',\delta) \geq \delta^\tau \Ncov(E,\delta)$.
\end{thm}
The $r = 1$ case is due to Bourgain~\cite{Bourgain2010} and is used in~\cite{BFLM}.
It is clear from the statement that $\tau$ can be chosen to be uniform for $\alpha$ varying between $\alphaini$ and $\alphahigh$.

\begin{proof}[Proof of Proposition~\ref{pr:prop6.3}]
Let $\kappa_0 = c$ and $m_0 = l_0$ be the constants given by Proposition~\ref{pr:RWonGr} applied to $\omega = 1$. We have, for all $m \geq m_0$
\begin{equation}
\label{eq:regularV+}
\forall \rho\in {[e^{-m}, e^{-m_0}]},\, \mu^{*m}\{g\in \Gamma \mid  \dang(V^+_g,W) \leq \rho\} \leq \rho^{\kappa_0/2}.
\end{equation}

Let $\tau > 0$ be the constant given by Theorem~\ref{thm:proj} applied to the parameters $\alpha$ and $\kappa = \min(\alphaini,\kappa_0/2)$. Put 
\[\omega = \frac{\tau \min(\lambda_1 - \lambda_{r+1},1)}{17d^2}.\]
Let $c > 0$ and $l_0 \geq 1$ be the constants depending on $\omega$ given by Proposition~\ref{pr:RWonGr}. Let be $C$ be a large constant and we assume that
\begin{equation}
\label{eq:assumC}
\log\frac{N}{M} \geq C \log\frac{2}{t} \text{ and } n \geq C\log\frac{N}{M}.
\end{equation}
The choice of $\omega$ allows us to choose an integer $m \leq C\log\frac{N}{M}$ such that 
\begin{equation}
\label{eq:assumm}
m \geq l_0,\, \frac{N}{M} \leq e^{(\lambda_1 - \lambda_{r+1} - \omega)m},\, \frac{N}{M} \leq e^m \text{ and } e^{16d^2\omega m} \leq \Bigl(\frac{N}{M}\Bigr)^\tau .
\end{equation}

Write $k = \frac{d}{r}$ and define
\[ \Glen = \bigl\{\ g \in \Gamma \mid \abs{\frac{1}{m}\log\sigma_j(g) - \lambda_j} \leq \omega \text{ for $j = 1,\dotsc,d$} \bigr\}\]
and
\[\Gvol = \bigl\{ (g_1,\dotsc g_k) \in \Gamma^k \mid \dang(V^+_{g^*_j}, V^+_{g^*_1} + \dotsb + V^+_{g^*_{j-1}}) \geq e^{-\omega m} \text{ for $j = 1,\dotsc,d$ } \bigr\}\]
It follows (from Proposition~\ref{pr:RWonGr}) that
\[\mu^{*m}(\Glen) \geq 1 - e^{-cm}\]
and
\[(\mu^{*m})^{\otimes k} (\Gvol) \geq 1 - k e^{-cm}\]

For $g \in \Gamma$, let $g^* = k\diag(\sigma_1(g),\dotsc,\sigma_d(g))l$ be a Cartan decomposition of its transposition. Define 
\[\theta_g = k\diag(e^{-\lambda_1 m}\sigma_1(g),\dotsc,e^{-\lambda_1 m}\sigma_r(g),1,\dotsc,1)l\]
and
\[\pi_g = l^{-1}\diag(\underbrace{1,\dotsc,1}_{\text{$r$ times}},0,\dotsc,0)l\]
so that $\pi_g$ is a rank $r$ orthogonal projection of image $\im(\pi_g) = V^+_g$ and $\theta_g\pi_g$ has image $V^+_{g^*}$. 
When $g \in \Glen$, we have
\begin{equation}
\label{eq:gthetapi}
\norm{g^* - e^{\lambda_1 m}\theta_g\pi_g} \leq \sigma_{r+1}(g) \leq e^{(\lambda_{r+1} + \omega) m}.
\end{equation}
Moreover, the part $\theta_g$ is almost orthogonal, that is.
\begin{equation}
\label{eq:thetaalO}
\norm{\theta_g}, \norm{\theta_g^{-1}} \leq  e^{\omega m}.
\end{equation}

By \cite[Lemma 6.7]{BFLM} applied to $\epsilon = \min(\frac{\tau}{20},\frac{d - \alphahigh}{6})$, we can find $N_1$ such that
\begin{equation}
\label{eq:rangeN1}
\frac{1}{2} \log\frac{N}{M} \leq \log\frac{N_1}{M} \leq \log\frac{N}{M}
\end{equation}
and a $M$-separated subset $A \subset A_{2^{-2}t^2,n}$ which is $(C_\epsilon t^{-2}, \alpha - \tau/2)$\dash{}regular at scale $M$, i.e.
\begin{equation}
\label{eq:regularA}
\forall s \geq M,\, \max_{x \in \R^d} \abs{A \cap B(x,s)} \leq C_\epsilon t^{-2} \bigl(\frac{s}{M}\bigr)^{\alpha - \tau/2} \abs{A}.
\end{equation}

By throwing away some elements, we may assume that $\hat\nu_n(a)$ with $a \in A$ all lie in a single quadrant of $\C$, so that
\begin{equation}
\label{eq:defAnu}
\abs{\sum_{a \in A} \hat\nu_n(a)} \geq 2^{-3}t^2 \abs{A}.
\end{equation}

Renormalize $A$ by setting $E = \frac{1}{N_1}\cdot A$. Let $\eta$ be image measure of $\mu^{*m}$ by the map $g \mapsto \im(\pi_g)$. Now apply Theorem~\ref{thm:proj} to the set $E$ and to the measure $\eta$ at scale $\delta = \frac{M}{N_1}$. The assumptions are readily satisfied from \eqref{eq:regularA} and \eqref{eq:regularV+} provided that
\[ C_\epsilon t^{-2} \leq \delta^{\tau/2}, e^{-m} \leq \delta \text{ and } \delta^\tau \leq e^{- m_0},\]
which is achieved under the assumptions~\eqref{eq:assumC} and \eqref{eq:assumm} with $C$ chosen large enough according to the other parameters.
We obtain a subset $\Gproj \subset \Gamma$ such that 
\[\mu^{*m}(\Gproj) \geq 1 - \delta^\tau\]
and for any $g \in \Gproj$ and any $A'\subset A$ with $\abs{A'} \geq \delta^\tau \abs{A}$, we have
\[\Ncov(\pi_g(A'),M) \geq \delta^{-\alpha/k - \tau}.\]
If moreover $g \in \Glen$, in view of \eqref{eq:thetaalO}, this yields
\[\Ncov(e^{\lambda_1 m}\theta_g \pi_g (A'),M')\geq \delta^{-\alpha/k - \tau} e^{-d\omega m}\]
where $M' = e^{\lambda_1 m}M$. From \eqref{eq:gthetapi} and \eqref{eq:assumm}, we have for all $a \in A'$, 
\[\norm{g^* a - e^{\lambda_1 m}\theta_g \pi_g a} \leq e^{(\lambda_{r+1} + \omega)m} N_1 \leq M'.\]
It follows that
\[g^* A' \subset \Nbd(V^+_{g^*}, M')\]
and
\[\Ncov(g^* A',M') \geq 2^{-d}\delta^{-\alpha/k - \tau} e^{-d\omega m} \geq \Bigl(\frac{N'}{M'}\Bigr)^{\alpha/k + \tau} e^{-2d\omega m} \]
where $N' = 2de^{(\lambda_1+ \omega) m}N_1$. 
Note that from the choice of $M'$ and $N'$, \eqref{eq:propN'M'} is satisfied.

Taking the $2k$-th power of \eqref{eq:defAnu} and applying the Hölder inequality, we obtain
\[ 2^{-6k}t^{2k} \abs{A}^{2k} \leq \iint_{\vec{g}, \vec{g}' \in \Gamma^k} \sum_{\vec{a}, \vec{a}' \in A^k} \abs{\hat\nu_{n'}(\Sigma_{\vec{g}}(\vec{a}) - \Sigma_{\vec{g}'}(\vec{a}'))} \dd (\mu^{*m})^{\otimes k}(\vec{g}) \dd (\mu^{*m})^{\otimes k}(\vec{g}')\]
where $\Sigma_{\vec{g}}(\vec{a}) = \sum_{i=1}^k g^*_ia_i$ for $\vec{g} = (g_1,\dotsc,g_k)$ and $\vec{a} = (a_1,\dotsc,a_k)$. By pigeonholing and the fact that
\[ 2^{-6k-1}t^{2k} \geq e^{-cm},\] 
we can find $\vec{g}' \in \Glen$ and $\vec{a}' \in \Gamma^k$ such that, writing $\mathbf{a} = \Sigma_{\vec{g}'}(\vec{a}')$, we have
\[\int_{\vec{g} \in \Gamma^k}\sum_{\vec{a} \in A^k}  \abs{\hat\nu_{n'}(\Sigma_{\vec{g}}(\vec{a}) -  \mathbf{a})} \dd (\mu^{*m})^{\otimes k}(\vec{g}) \geq 2^{-6k-2}t^{2k}\abs{A}^k.\]
Again by pigeonholing, the set $\Gstat$ defined by
\[\Gstat = \bigl\{\vec{g} \in \Gamma^k \mid \sum_{\vec{a} \in A^k} \abs{\hat\nu_{n'}(\Sigma_{\vec{g}}(\vec{a}) -  \mathbf{a})} \geq 2^{-6k-3}t^{2k}\abs{A}^k \bigr\}\]
satisfies $(\mu^{*m})^{\otimes k}(\Gstat) \geq 2^{-6k-3}t^{2k}$.
Then the set $\mathcal{G} =  \Glen^k \cap \Gvol \cap \Gproj^k \cap \Gstat$ has measure
\[\mu^{*m}(\mathcal{G}) \geq 2^{-6k-3}t^{2k} - 2k e^{-cm} - k\delta^\tau > 0\]
provided that $C$ is chosen large enough in \eqref{eq:assumC}.

From now on fix an element $\vec{g} \in \mathcal{G}$. 
By iterative pigeonholing we can find a subset $T \subset A^k$ having the following properties 
\begin{enumerate}
\item We have $\Sigma_{\vec{g}}(T) - \mathbf{a} \subset A_{t',n'}$ with $t'= 2^{-10k}t^{2k}$ and $n'= n -m$.
\item There is a tree structure associated to $T$ for which $T$ is the set of leaves. Namely, the tree have $k + 1$ level. Level $0$ is the root and for $j = 1,\dotsc, k$, the level $i$ vertices are
\[T_j = \{(a_1,\dotsc, a_j) \in A^j \mid \exists (a_{j+1},\dotsc,a_k) \in A^{k-j}, (a_1,\dotsc,a_k) \in T\}.\]
For each node $(a_1,\dotsc, a_{j-1}) \in T_{j-1}$, its descendants are the elements in $T_j$ that have $(a_1,\dotsc, a_{j-1})$ as the $j-1$ first coordinates. Thus their $j$-th coordinates are
\[A_j(a_1,\dotsc, a_{j-1}) =\{a_j \in A \mid (a_1,\dotsc, a_j) \in T_j\}.\]
\item For all $j = 1,\dotsc, k$ and all $t \in T_{j-1}$
\[\abs{A_j(t)} \geq 2^{-10k}t^{2k} \abs{A} \geq \delta^\tau \abs{A}.\]
\end{enumerate}

From the last point and the fact $g_j \in \Glen \cap \Gproj$, we know that for all $j = 1,\dotsc, k$ and all $t \in T_{j-1}$,
\begin{equation}
\label{eq:gAinV+g}
g^*_j A_j(t) \subset \Nbd(V^+_{g^*_j}, M')
\end{equation}
and
\begin{equation}
\label{eq:gAbig}
\Ncov(g^*_j A_j(t),M') \geq \Bigl(\frac{N'}{M'}\Bigr)^{\alpha/k + \tau} e^{-2d\omega m}.
\end{equation}

It is also clear that $\Sigma_{\vec{g}}(T) - \mathbf{a}$ is contained in $B(0,N')$. We will put \eqref{eq:gAbig} for different $j$ together to establish
\[\Ncov(\Sigma_{\vec{g}}(T) , M') \geq \Bigl(\frac{N'}{M'}\Bigr)^{\alpha + k\tau}e^{-4kd\omega m} .\]
which finish the proof the proposition after taking \eqref{eq:assumm} into account.

For that it suffices to prove for $j=1, \dotsc, k$, 
\begin{equation}
\label{eq:inductionTi}
\Ncov(\Sigma^{(j)}_{\vec{g}}(T_j),M') \geq  \Bigl(\frac{N'}{M'}\Bigr)^{j\alpha/k + j\tau} e^{-4d\omega m} \Ncov(\Sigma^{(j-1)}_{\vec{g}}(T_{j-1}),M')
\end{equation}
where $\Sigma^{(j)}_{\vec{g}} \colon T_j \to \Z^d$ is a shorthand for
\[\Sigma^{(j)}_{\vec{g}}(a_1,\dotsc,a_j) = \sum_{i=1}^j g^*_i a_i\]
with the convention $\Sigma^{(0)}_{\vec{g}} = 0$.

First, by \eqref{eq:gAinV+g}, for $j=1, \dotsc, k$, 
\[\Sigma^{(j)}_{\vec{g}}(T_j) \subset \Nbd(W_j, j M')\]
where $W_j \subset \Gr(jr,d)$ stands for
\[W_j = \bigoplus_{i=1}^{j} V^+_{g^*_i}.\]
Since $\vec{g} \in \Gvol$, we have
\begin{equation}
\label{eq:dangV+W}
\dang(V^+_{g^*_j}, W_{j-1}) \geq e^{-\omega m}
\end{equation}

From the definition of $M'$-covering number, we can find a subset $T'_{j-1} \subset T_{j-1}$ such that 
\[\abs{T'_{j-1}} \geq (8k)^{-d}e^{-d\omega m} \Ncov(\Sigma^{(j-1)}_{\vec{g}}(T_{j-1}),M'),\]
the map $\Sigma^{(j-1)}_{\vec{g}}$ restricted to $T'_{j-1}$ is injective and has $4ke^{\omega m} M'$-separated image.

From \eqref{eq:dangV+W} it follows that the $M'$-neighborhoods of the sets
\[\Sigma^{(j-1)}_{\vec{g}}(t) + g^*_j A_j(t),\quad t \in T'_{j-1}\]
are pairwise disjoint. Hence, taking \eqref{eq:gAbig} into account,
\begin{align*}
\Ncov(\Sigma^{(j)}_{\vec{g}}(T_j),M') &\geq \sum_{t \in T'_{j-1}} \Ncov(\Sigma^{(j-1)}_{\vec{g}}(t) + g^*_j A_j(t), M')\\
&\geq \sum_{t \in T'_{j-1}} \Ncov(g^*_j A_j(t), M')\\
&\geq \abs{T'_{j-1}} \Bigl(\frac{N'}{M'}\Bigr)^{\alpha/k + \tau} e^{-2d\omega m},
\end{align*}
establishing \eqref{eq:inductionTi}.
\end{proof}

\subsection{From high dimension to positive density.} 
In the last step of Phase I (Proposition 6.5 in \cite{BFLM}), one gets positive density of $A_{t,n}$, again by paying the price of making $t$ and $n$ smaller. 
The argument is essentially a discretized version of a projection theorem due to Peres-Schlag~\cite[Proposition 6.1]{PeresSchlag} combined with the same proof above for Proposition~\ref{pr:prop6.3} (Proposition 6.3 in~\cite{BFLM}). 

The projection theorem of Peres-Schlag~\cite[Proposition 6.1]{PeresSchlag} is valid for projections of all rank. The $r = 1$ case was used in~\cite{BFLM}. For our purpose, we need the general case. 

Let us state the discretized version and then indicate its proof.
\begin{prop}[Discretized version of {\cite[Proposition 6.1]{PeresSchlag}}]
\label{pr:PeresSchlag6.1}
Let $0 < r < d$. Given parameters $\alpha,\beta, C_E,C_\eta > 0$ such that $\alpha + \beta/4 > d$, there exists a constant $C = O_{d,\beta}(C_E C_\eta)$ such that the following hold.  
Let $E \subset B(0,1)$ be a subset of $\R^d$ and $\eta$ be a probability measure on $\Gr(r,d)$.
Assume that $E$ is $\delta$\dash{}separated and
\[\forall \rho \geq \delta,\, \forall x \in \R^d,\; \abs{E \cap B(x,\rho)} \leq C_E \rho^\alpha \abs{E}\]
and
\begin{equation}
\label{eq:etaNCxV}
\forall \rho \geq \delta, \, \forall x \in \mathbb{P}(\R^d),\;  \eta \{ V \in \Gr(r,d) \mid \dang(x,V) \leq \rho \} \leq C_\eta \rho^\beta
\end{equation}
Then there exists $V \in \Supp(\eta)$ such that for all subset $E' \subset E$ 
\[\Ncov(\pi_V(E'), \delta) \geq \frac{\abs{E'}^2}{C\abs{E}^2}\delta^{-r}.\]
Consequently, for any $C' \geq C$, there exists $\mathcal{D} \subset \Gr(r,d)$ such that $\eta(\mathcal{D}) \geq 1 - \frac{C}{C'}$ and such that for all $V \in \mathcal{D}$, we have
\begin{equation}
\label{eq:piVE'C0}
\forall E' \subset E,\;\Ncov(\pi_V(E'), \delta) \geq \frac{\abs{E'}^2}{C'\abs{E}^2}\delta^{-r}.
\end{equation}
\end{prop}
Remark that the assumption~\eqref{eq:etaNCxV} will be guaranteed by Proposition~\ref{pr:RWonGr} when Proposition~\ref{pr:PeresSchlag6.1} is used. 
In the remainder of this subsection, we will sketch the proof of Proposition~\ref{pr:PeresSchlag6.1}.

Let $\Phi$ be a radially symmetric nonnegative smooth function on $\R^d$ with $\norm{\Phi}_1 = 1$ and supported on $B(0,1)$. Set for $\delta > 0$,
\[\Phi_\delta(x) = \delta^{-d}\Phi(\delta^{-1}x).\]
Let $\nu$ be a measure on $\R^d$ and $V \in \Gr(r,d)$. We denote by $\nu_V$  the push-forward of $\nu$ by the orthogonal projection $\pi_V$. Its Fourier transform satisfy
\[\forall \xi \in \R^d, \; \hat{\nu}_V(\xi) = \hat{\nu}(\pi_V(\xi)).\]

For $k = 0, \dotsc, d$, denote by $\mathcal{H}^k$ the $k$-dimensional Hausdorff measure on $\R^d$. We define $\Phi_{\delta,V} \colon V \to \R_+$ to be the Radon-Nikodym derivative of $( \Phi_\delta \mathcal{H}^d)_V$ with respect to $\mathcal{H}^r$. Let $\hat{\Phi}_{\delta,V}$ denote the Fourier transform of the measure $\Phi_{\delta,V} \mathcal{H}^r$. In other words,
\[\forall \xi \in \R^d,\; \hat{\Phi}_{\delta,V}(\xi) = \hat{\Phi}_{\delta} (\pi_V(\xi)).\]

The following inequality is the discretized version of the inequality at the heart of \cite[Proposition 6.1]{PeresSchlag}. The proof is essentially the same with little adaptation needed. This adaptation to scale $\delta$ is explained in detail in \cite[Proposition 6.11]{BFLM}, which dealt the case $r=1$. 
\begin{prop}
\label{pr:PeresSchlagIne}
Let any $\beta > 0$ and $\delta > 0$ the following holds. Let $\nu$ be a Borel measure on $\R^d$ and $\eta$ a Borel probability measure on $\Gr(r,d)$. 
Assume that there exists $C_\eta > 0$ such that for all $x \in \mathbb{P}(\R^d)$,
\[\forall \rho \geq \delta,\quad \eta\{V \in \Gr(r,d) \mid \dang(x,V) \leq \rho\} \leq C_\eta \rho^\beta.\]
Then
\begin{multline}
\int_{\Gr(r,d)} \int_V \abs{\hat{\nu}_V(\xi)}^2 \abs{\hat{\Phi}_{\delta,V}(\xi)}^2 \dd\mathcal{H}^k(\xi) \dd \eta(V)\\
\leq C_\eta C_d \int_{\R^d} \abs{\hat\nu(x)}^2 \abs{\hat\Phi_\delta(x)}^2 (1 + \abs{x})^{-\beta/2} \dd x + C_\eta C_d
\end{multline}
where $C_d > 0$ is a constant depending only on $d$.
\end{prop}

\begin{proof}[{Proof of Proposition~\ref{pr:PeresSchlag6.1}}]
Let $\nu$ be the uniform probability measure supported on $E$. Using the terminology of \cite[Definition 5.1]{BFLM}, the assumption on $E$ tells us that $\nu$ is $(C_E,\alpha)$-regular at scale $\delta$ on $B(0,1)$. Using the notion of $\alpha$-energy and its Fourier interpretation\cite[12.12]{Mattila95}, we have (see the discussion at the end of Section 5 in~\cite{BFLM})
\[\int_{\R^d} \abs{\hat\nu(x)}^2 \abs{\hat\Phi_\delta(x)}^2 (1 + \abs{x})^{\alpha - \beta/4 - d} \dd x \ll_{d,\beta} C_E\]
Taking into account the assumption that $\alpha + \beta/4 \geq d$, the Plancherel Theorem and Proposition~\ref{pr:PeresSchlagIne}, we have
\[\int_{\Gr(r,d)} \norm{ \nu_V * \Phi_{\delta,V}}_{L^2(V, \mathcal{H}^r)}^2 \dd\eta(V) \ll_{d,\beta}C_E C_\eta.\]
Thus there exists $V \in \Supp(\eta)$ such that
\[\norm{ \nu_V * \Phi_{\delta,V}}_{L^2(V, \mathcal{H}^r)}^2  \ll_{d,\beta}C_E C_\eta.\]
By Cauchy-Schwarz (see \cite[Lemma 6.10]{BFLM}), for any subset $E' \subset E$,
\[\Ncov(\pi_V(E'),\delta) \gg_r  \delta^{-r} \norm{ \nu_V * \Phi_{\delta,V}}_{L^2(V, \mathcal{H}^r)}^{-2} \nu(E')^2,\]
which finishes the proof of the first part of Proposition~\ref{pr:PeresSchlag6.1}.

The "consequently" part is obtained by applying the first part to the restriction of $\eta$ to the set of all $V \in \Gr(r,d)$ such that \eqref{eq:piVE'C0} does not hold.
\end{proof}

\subsection{Phase II: Granulated measures.}
In Phase II (Section 7 in \cite{BFLM}), the proximality assumption {\upshape (P)} is used in several places. However there is no difficulty in adapting the proof under our assumptions~\eqref{eq:propertyS} and that $r$ divides $d$.

For example, the argument in \cite{BFLM} requires to find $d$-tuple $(g_1,\dotsc,g_d)$ of elements of $\Gamma$ such that 
\[\vol( \theta(g_1), \dotsc, \theta(g_d)) \]
is large, where $\theta(g)$ is the direction of the largest axis of the ellipsoid $g(B(0,1))$ and for nonzero $x_1, \dotsc , x_d \in \R^d$,
\[\vol(\bar{x}_1, \dotsc, \bar{x}_d) = \frac{\abs{\det(x_1, \dotsc, x_d)}}{\norm{x_1} \dotsm \norm{x_d}}.\]

For our situation, define for subspaces $V_1, \dotsc, V_k$ of $\R^d$,
\[\dang(V_1, \dotsc, V_k) = \frac{\norm{\mathbf{v_1} \wedge \dotsm \wedge \mathbf{v_j}}}{\norm{\mathbf{v_1}} \dotsm \norm{\mathbf{v_k}}}\]
where $\mathbf{v_i}$ is the wedge product of a basis of $V_i$ for every $i = 1, \dotsc, k$.
Equivalently,
\[\dang(V_1, \dotsc, V_k) = \prod_{j=2}^k \dang(V_j, V_1 + \dotsc + V_{j-1}).\]

Write  $k = d / r$, we want to find $k$-tuple $(g_1,\dotsc,g_k)$ of elements of $\Gamma$ such that $V^+_{g_1}, \dotsc, V^+_{g_k}$ are well spaced.
More precisely, using Proposition~\ref{pr:RWonGr} for $k-1$ times, we obtain $c > 0$ depending on $\omega > 0$ such that for all $n$ large enough,
\[(\mu^{*n})^{\otimes k}\bigl\{ (g_1,\dotsc,g_k) \mid \dang(V^+_{g_1}, \dotsc, V^+_{g_k}) \geq e^{-\omega n} \bigr\} \geq 1 - e^{-cn}.\]

Having this change in mind, the proof in \cite[Section 7]{BFLM} works almost verbatim. Let us only explain in detail one of the steps.
The following is the analogue for $r \geq 2$ of Lemma 7.9 in \cite{BFLM}. 
\begin{lemm}
Let $d = k r$ be integers. Given $g_1,\dotsc,g_k \in \GL(\R^d)$ and constants $c_1, \dotsc, c_k \in \R$, let
\[\rho = \max_{1 \leq i \leq k} \frac{\sigma_{r+1}(g_i)}{\sigma_r(g_i)},\quad C = \max_{1\leq i, j\leq k} \frac{\abs{c_i}}{\abs{c_j}}, \quad L = C = \max_{1\leq i, j\leq k} \frac{\sigma_1(g_i)}{\sigma_r(g_j)}\]
and let $v= min\{v_1,v_2\}$ where
\[v_1 = \dang(V^+_{g_1},\dotsc,V^+_{g_k}), \text{ and } v_2 = \dang(V^+_{g^* _1},\dotsc,V^+_{g^* _k}).\]
Assume that 
\[\rho < \frac{v^3}{40k^3CL}.\]
Then the matrix $h = \sum_{i=1}^k c_i g_i$ is invertible.
\end{lemm}

The proof follow the same pattern as in \cite[Lemma 7.9]{BFLM}. We will need a few basic estimates. They are natural and straightforward generalisations of Lemma 4.1 and Lemma 7.7 in \cite{BFLM}.
For any $g \in \GL(\R^d)$ and any $x \in \R^d \setminus\{0\}$. 
\begin{equation}
\label{eq:danggxx}
\sigma_r(g) \dang(\bar{x},V^-_g) \leq \frac{\norm{g x}}{\norm{x}} \leq \sigma_1(g)\dang(\bar{x},V^-_g) + \sigma_{r+1}(g).
\end{equation}
\begin{equation}
\label{eq:dangdang}
\dang(\bar{x},V^-_g) \dang(g\bar{x}, V^+_g) \leq \frac{\sigma_{r+1}(g)}{\sigma_r(g)}.
\end{equation}
For any $1 \leq l \leq d$ and any nonzero vectors $x_1,\dotsc, x_l, y_1,\dotsc,y_l \in \R^d$,
\begin{equation}
\label{eq:dangww'}
\abs{\dang(\bar{x}_1,\dotsc,\bar{x}_l) -  \dang(\bar{y}_1,\dotsc,\bar{y}_l)} \leq 2 \sum_{i=1}^l \dang(\bar{x}_i,\bar{y}_i).
\end{equation}

\begin{proof}
Let $x \in \R^d$ be a unit vector. Rearrange the $g_i$'s so that
\[\alpha_i = \dang(\bar{x},V^-_g)\]
is nonincreasing in $i$. Put $\beta = 4k\rho/v$ and define $l = \max\{1 \leq i \leq k \mid \alpha_i > \beta\}$. Writing $x_i = c_i g_i z$, we shall prove
\begin{equation}
\label{eq:x1xlbig}
\norm{x_1 + \dotsb + x_l}  > \norm{x_{l+1}} + \dotsb + \norm{x_{k}},    
\end{equation}
which will imply that $h x \neq 0$, showing that $h$ is invertible.

First, we show that $\alpha_1 \geq v/2k$, thus ensuring that $k \geq 1$. Indeed, for each $i = 1 ,\dotsc, k$, there is $y_i \in V^-_{g_i}$ such that $\dang(\bar{x},\bar{y}_i) \leq \alpha_i$. Inside the plane spanned by $x$ and $y_i$ we can find $x'_i \in \bar{x}^\perp$ and $y'_i \in \bar{y}_i^\perp$ such that
\[\dang(\bar{x}'_i,\bar{y}'_i) = \dang(\bar{x},\bar{y}_i) \leq \alpha_i.\]
We have 
\[\dang(\bar{x}'_i,\dotsc,\bar{x}'_k) = 0
 \text{ and }
\dang(V^+_{g^* _1},\dotsc,V^+_{g^* _k}) \leq \dang(\bar{y}'_i,\dotsc,\bar{y}'_k)\]
because $\forall i$, $x'_i \in \bar{x}^\perp$ and $y'_i \in \bar{y}_i^\perp \subset (V^-_{g_i})^\perp = V^+_{g^* _i}$.
By the assumption and \eqref{eq:dangww'}, we have
\[v \leq 2 \sum_{i=1}^k \alpha_i.\]
Therefore $\alpha_1 \geq v/2k$.

Now we bound from below the quantity $\norm{x_1 + \dotsb + x_l}$. For each $i = 1 ,\dotsc, l$, pick $z_i \in V^+_{g_i}$ such that $\dang(\bar{x}_i, \bar{z}_i) = \dang(\bar{x}_i, V^+_{g_i})$. From \eqref{eq:dangdang} and the definition of $l$, it follows that
\[\dang(\bar{x}_i,\bar{z}_i) \leq \frac{\rho}{\beta}.\]
Hence
\begin{align*}
\dang(\bar{x}_1,\dotsc,\bar{x}_l) &\geq \dang(\bar{z}_1,\dotsc,\bar{z}_l)  - 2l\rho \beta^{-1}\\
&\geq \dang( V^+_{g_1},\dotsc,  V^+_{g_l}) - 2l\rho \beta^{-1}\\
&\geq v - 2l\rho \beta^{-1}.
\end{align*}
On the other hand,
\begin{align*}
\dang(\bar{x}_1,\dotsc,\bar{x}_l) &= \frac{\norm{x_1 \wedge \dotsm \wedge x_l}}{\norm{x_1} \dotsm \norm{x_l}}\\
&= \frac{\norm{(x_1 + \dotsb + x_l) \wedge x_2 \wedge \dotsm \wedge x_l}}{\norm{x_1} \dotsm \norm{x_l}}\\
&\leq \frac{\norm{x_1 + \dotsb + x_l}}{\norm{x_1}}
\end{align*}
Moreover, by \eqref{eq:danggxx},
\[\norm{x_1} = \norm{c_1 g_1 x} \geq \abs{c_1} \sigma_r(g_1) \dang(\bar{x},V^-_{g_1}) \geq \abs{c_1} \sigma_r(g_1) \alpha_1.\]
Hence 
\[\norm{x_1 + \dotsb + x_l} \geq (v - 2l\rho \beta^{-1}) \alpha_1 \abs{c_1} \sigma_r(g_1).\]

Finally we will bound from above each of the quantities $\norm{x_j}$ for $j = l+1, \dotsc, k $. By \eqref{eq:danggxx},
\begin{align*}
\norm{x_j} = \norm{c_j g_j x} &\leq \abs{c_j} \bigl(\sigma_1(g_j) \dang(\bar{x},V^-_{g_j}) + \sigma_{r+1}(g_j)\bigr)\\
&\leq C L \abs{c_1} \sigma_r(g_1)(\beta + \rho).
\end{align*}
This proves \eqref{eq:x1xlbig} under our assumption on $\rho$.
\end{proof}